\documentclass{amsart}
\usepackage{amssymb,amsthm,amstext,amsmath,amscd,latexsym,amsfonts,amscd}
\usepackage{color}
\usepackage{setspace}
\usepackage{enumerate}
\usepackage{braket}
\usepackage[all]{xy}
\usepackage{multirow}
\usepackage{mathrsfs}
\usepackage{graphicx}
\usepackage{enumitem}
\usepackage{ulem}
\newtheorem{theorem}[equation]{Theorem}
\newtheorem{proposition}[equation]{Proposition}
\newtheorem{lemma}[equation]{Lemma}
\newtheorem{corollary}[equation]{Corollary}
\theoremstyle{definition}
\newtheorem{definition}[equation]{Definition}
\newtheorem{remark}[equation]{Remark}
\newtheorem{notation}[equation]{Notation}
\newtheorem{example}[equation]{Example}

\theoremstyle{plain}

\newcommand{\disp}{\displaystyle}

\newcommand{\Spec}{\mathop{\mathrm{Spec}}\nolimits}

\newcommand{\supp}{\mathop{\mathrm{supp}}\nolimits}
\newcommand{\cosupp}{\mathop{\mathrm{cosupp}}\nolimits}
\newcommand{\id}{\mathop{\mathrm{id}}\nolimits}
\newcommand{\Ker}{\mathop{\mathrm{Ker}}\nolimits}

\newcommand{\Image}{\mathop{\mathrm{Im}}\nolimits}

\newcommand{\Ext}{\mathop{\mathrm{Ext}}\nolimits}

\newcommand{\Hom}{\mathop{\mathrm{Hom}}\nolimits}
\newcommand{\RHom}{\mathop{\mathrm{RHom}}\nolimits}
\newcommand{\RGamma}{\mathop{\mathrm{R\Gamma}}\nolimits}
\newcommand{\LLambda}{\mathop{\mathrm{L\Lambda}}\nolimits}
\newcommand{\Mod}{\mathop{\mathrm{Mod}}\nolimits}

\newcommand{\tot}{\mathop{\mathrm{tot}}\nolimits}

\newcommand{\Lotimes}{\mathop{\otimes_R^{\rm L}}\nolimits}

\newcommand{\fa}{\mathfrak{a}}

\newcommand{\fp}{\mathfrak{p}}

\newcommand{\fq}{\mathfrak{q}}

\newcommand{\fm}{\mathfrak{m}}

\newcommand{\cD}{\mathcal{D}}

\newcommand{\cL}{\mathcal{L}}
\newcommand{\cC}{\mathcal{C}}
\newcommand{\cT}{\mathcal{T}}

\newcommand{\Z}{\mathbb{Z}}

\makeatletter
  
  \@addtoreset{equation}{section}
\makeatother
\begin{document}
\title[Localization functors and cosupport]{Localization functors and cosupport in derived categories of commutative Noetherian rings}
\author[T. Nakamura]{Tsutomu Nakamura}
\address[T. Nakamura]{Graduate School of Natural Science and Technology, Okayama University,
Okayama, 700-8530 Japan}
\email{t.nakamura@s.okayama-u.ac.jp}
\author[Y. Yoshino]{Yuji Yoshino}
\address[Y. Yoshino]{Graduate School of Natural Science and Technology, Okayama University,
Okayama, 700-8530 Japan}
\email{yoshino@math.okayama-u.ac.jp}
\begin{abstract} Let $R$ be a commutative Noetherian ring.
We introduce the notion of localization functors  $\lambda^W$ with cosupports in arbitrary subsets $W$ of $\Spec R$; it is a common generalization of  localizations with respect to multiplicatively closed subsets and left derived functors of ideal-adic completion functors.
We prove several results about the localization functors $\lambda^W$, including an explicit way to calculate $\lambda^W$ by the notion of \v{C}ech complexes.
As an application, we can give a simpler proof of a classical theorem by Gruson and Raynaud, which states that the projective dimension of a flat $R$-module is at most the Krull dimension of $R$.
As another application, it is possible to give a functorial way to replace complexes of flat $R$-modules or  complexes of finitely generated $R$-modules by complexes of pure-injective $R$-modules.
\end{abstract}

\maketitle
\tableofcontents
\section{Introduction}
\label{Introduction}
Throughout this paper, we assume that $R$ is a commutative Noetherian ring.
We denote by $\cD=D(\Mod R)$ the derived category of all complexes of $R$-modules, by which we mean that $\cD$ is the unbounded derived category.
For a triangulated subcategory $\cT$ of $\cD$, its left (resp. right) orthogonal subcategory is defined as ${}^\perp\cT=\Set{X\in \cD | \Hom_\cD(X,\cT)=0}$
 (resp. $\cT^\perp=\Set{Y\in \cD | \Hom_\cD(\cT,Y)=0}$).
Moreover, $\cT$ is called localizing (resp. colocalizing) if $\cT$ is closed under arbitrary direct sums (resp. direct products).

Recall that the support of a complex $X\in \cD$ is defined as follows;
\begin{align*}
\supp X&= \Set{ \fp \in \Spec R | X\Lotimes \kappa(\fp)\neq 0},
\end{align*}
where $\kappa(\fp)=R_\fp/\fp R_\fp$. We write  
$\cL_W=\Set{X\in \cD | \supp X\subseteq W }$ for a subset $W$ of $\Spec R$. 
Then $\cL_W$ is a localizing subcategory of $\cD$.
Neeman \cite{N} proved that any localizing subcategory of $\cD$ is obtained in this way.
The localization theory of triangulated categories \cite{K} yields a couple of adjoint pairs $(i_W, \gamma_W)$ and $(\lambda_W, j_W)$ as it is indicated in the following diagram:
\begin{align}\text{
$
\xymatrix{
\cL_W\ \ar@<3.5pt>[rr]^*{i_W}&&\ar@<3.5pt>[ll]^*{\gamma_W}\ \cD\ \ar@<3.5pt>[rr]^*{\lambda_W}&&\ar@<3.5pt>[ll]^*{j_W}\ \cL_W^\perp}
$}\label{adjoint diagram 1}
\end{align}
Here, $i_W$ and $j_W$ are the inclusion functors $\cL_W\hookrightarrow \cD$ and $\cL_W^\perp\hookrightarrow \cD$ respectively.
In the precedent paper \cite{NY}, the authors introduced the colocalization functor with support in $W$ as the functor $\gamma_W$.
If $V$ is a specialization-closed subset of $\Spec R$, then $\gamma_V$ coincides with the right derived functor $\RGamma_V$ of the section functor $\Gamma_V$ with support in $V$; it induces the local cohomology functors $H^i_V(-)=H^i(\RGamma_V(-))$.
In {\it loc.\ cit.}, we established some methods to compute $\gamma_W$ for general subsets $W$ of $\Spec R$.
Furthermore, the local duality theorem and Grothendieck type vanishing theorem of local cohomology were  extended to the case of  $\gamma_W$.

On the other hand, in this paper, we introduce the notion of localizations functors with cosupports in arbitrary subsets $W$ of $\Spec R$.
Recall that the cosupport of a complex $X\in \cD$ is defined as follows;
$$\cosupp X=\Set{\fp\in \Spec R | \RHom_R(\kappa(\fp),X)\neq 0}.$$
We write $\cC^W=\Set{X\in \cD | \cosupp X\subseteq W }$ for a subset $W$ of $\Spec R$. 
Then $\cC^W$ is a colocalizing subcategory of $\cD$.
Neeman \cite{N3} proved that any colocalizing subcategory of $\cD$ is obtained in this way.\footnotemark{}
\footnotetext{This result is not needed in this work.}

We remark that there are equalities
\begin{align}
{}^\perp\cC^{W}=\cL_{W^c},\ \ \cC^{W}=\cL_{W^c}^\perp
,
\label{star}
\end{align}
where $W^c=\Spec R\backslash W$.
The second equality follows from Neeman's theorem  \cite[Theorem 2.8]{N}, which states that $\cL_{W^c}$ is equal to the smallest  localizing subcategory of $\cD$ containing the set $\Set{\kappa(\fp)|\fp\in W^c}$.
Then it is seen that the first one holds, since ${}{}^\perp(\cL_{W^c}^\perp)=\cL_{W^c}$ (cf. \cite[\S 4.9]{K}).

Now we write $\lambda^W=\lambda_{W^c}$ and $j^W=j_{W^c}$.
By (\ref{adjoint diagram 1}) and (\ref{star}), there is a diagram of adjoint pairs: 
$$
\xymatrix{
{}^\perp\cC^{W}=\cL_{W^c}\ \ar@<3.5pt>[rr]^*{\hspace{10mm}i_{W^c}}&&\ar@<3.5pt>[ll]^*{\hspace{10mm}\gamma_{W^c}}\ \cD\ \ar@<3.5pt>[rr]^*{\hspace{-4.8mm}\lambda^{W}}&&\ar@<3.5pt>[ll]^*{\hspace{-5.1mm}j^W}\ \cC^{W}=\cL_{W^c}^\perp
}
$$
We call $\lambda^W$ {\it the localization functor with cosupport in $W$}.

For a multiplicatively closed subset $S$ of $R$, the localization functor $\lambda^{U_S}$ with cosupport in $U_S$ is nothing but 
$(-)\otimes_RS^{-1}R$, where $U_S=\Set{\fp\in \Spec R | \fp\cap S=\emptyset}$.
Moreover, for an ideal $\fa$ of $R$, the localization functor $\lambda^{V(\fa)}$ with cosupport in $V(\fa)$ is isomorphic to the left derived functor $\LLambda^{V(\fa)}$ of the $\fa$-adic completion functor $\Lambda^{V(\fa)}=\varprojlim(-\otimes_R R/\fa^n)$ defined on $\Mod R$. See Section 2 for details.\medskip

In this paper, we establish several results about the localization functor $\lambda^W$ with cosupport in a general subset $W$ of $\Spec R$.

In Section 3, we prove that $\lambda^W$ is isomorphic to $\prod_{\fp\in W}\LLambda^{V(\fp)}(-\otimes_RR_\fp)$ if  there is no inclusion relation between two distinct prime ideals in $W$.
Furthermore, we give a method to compute $\lambda^W$ for a general subset $W$.
We write $\eta^{W}:\id_\cD\to \lambda^{W}(=j^W\lambda^W)$ for the natural morphism given by the adjointness of $(\lambda^W, j^W)$.
In addition, note that when $W_0\subseteq W$,  there is a morphism $\eta^{W_0}\lambda^W:\lambda^W\to\lambda^{W_0}\lambda^{W}\cong \lambda^{W_0}$.
The following theorem is one of the main results of this paper.

\begin{theorem}[{Theorem \ref{main result 1}}]\label{method for induction 1 intro}
 Let $W$, $W_0$ and $W_1$  be subsets of $\Spec R$ with $W=W_0\cup W_1$.
We denote by $\overline{W_0}^s$ (resp. $\overline{W_1}^g$) is the specialization (resp. generalization) closure of $W$.
Suppose that one of the following conditions holds:
\begin{itemize}
\item[\rm (1)] $W_0=\overline{W_0}^s\cap W$;
\item[\rm (2)] $W_1=W\cap \overline{W_1}^g$.
\end{itemize}
Then, for any $X\in \cD$, there is a triangle
 $$\begin{CD}
\lambda^{W} X@>f>>\lambda^{W_1}X\oplus\lambda^{W_0} X @>g>>  \lambda^{W_1}\lambda^{W_0} X @>>> \lambda^{W} X[1],
\end{CD}$$
where 
$$f=\left(\begin{array}{c}
\eta^{W_1}\lambda^{W}X\\[7pt]
\eta^{W_0}\lambda^WX
\end{array}\right),\ \  
g=\left(\begin{array}{cc}\lambda^{W_1}\eta^{W_0}X\hspace{10pt}
(-1)\cdot\eta^{W_1}\lambda^{W_0}X
\end{array}\right).$$
\end{theorem}\vspace{2mm}

This theorem enables us to compute $\lambda^W$ by  using $\lambda^{W_0}$ and $\lambda^{W_1}$ for smaller subsets $W_0$ and $W_1$.
Furthermore, as long as we consider the derived category $\cD$, the theorem and Theorem \ref{gamma pushout} in Section 3 generalize Mayer-Vietoris triangles by Benson, Iyengar and Krause \cite[Theorem 7.5]{BIK}.

In Section 4, as an application, we give a simpler proof of a classical theorem due to Gruson and Raynaud. The theorem states that the projective dimension of a flat $R$-module is at most the Krull dimension of $R$.

Section 5 contains some basic facts about cotorsion flat $R$-modules.

Section 6 is devoted to study the cosupport of a complex $X$ consisting of cotorsion flat $R$-modules.
As a consequence, we can calculate $\gamma_{V^c}X$ and $\lambda^VX$ explicitly for a specialization-closed subset $V$ of $\Spec R$.

In Section 7, using Theorem \ref{method for induction 1 intro} above, we give a new way to get $\lambda^{W}$.
In fact, provided that $d=\dim R$ is finite, we are able to calculate $\lambda^W$ by a ${\rm \check{C}}$ech complex of functors of the following form;
$$\begin{CD}
\disp{\prod_{0\leq i\leq d}\bar{\lambda}^{W_{i}}}@>>>\disp{\prod_{0\leq i<j\leq d}\bar{\lambda}^{W_{j}}\bar{\lambda}^{W_{i}}}@>>>\cdots @>>>   \disp{\bar{\lambda}^{W_d}\cdots\bar{\lambda}^{W_0}},
\end{CD}$$
where $W_i=\{\fp\in W | \dim R/\fp=i \}$ and $\bar{\lambda}^{W_i}=\prod_{\fp\in W_i}\Lambda^{V(\fp)}(-\otimes_RR_\fp)$ for $0\leq i\leq d$.
This \v{C}ech complex sends a complex $X$ of $R$-module to a double complex in a natural way. We shall prove that $\lambda^WX$ is isomorphic to the total complex of the double complex if $X$ consists of  flat $R$-modules.

Section 8 treats commutativity of $\lambda^W$ with tensor products. Consequently, we show that $\lambda^WY$ can be computed by using the \v{C}ech complex above if $Y$ is a complex of finitely generated $R$-modules.

In Section 9, as an application, we give a functorial way to construct quasi-isomorphisms from complexes of  flat $R$-modules or complexes of finitely generated $R$-modules to complexes of pure-injective $R$-modules.\\

\vspace{10pt}
\noindent{\bf Acknowledgements.}
The first author is grateful to Srikanth Iyengar for his helpful comments and suggestions.
The second author was supported by JSPS Grant-in-Aid for Scientific Research 26287008.

\section{Localization functors}
\label{Flat }

In this section, we summarize some notions and basic facts used in the later sections.\medskip

We write $\Mod R$ for the category of all modules over a commutative Noetherian ring $R$.
For an ideal $\fa$ of $R$, $\Lambda^{V(\fa)}$ denotes the $\fa$-adic completion functor $\varprojlim (-\otimes_R R/\fa^n)$ defined on $\Mod R$.
Moreover, we also denote by ${M}^{\wedge}_\fa$ the $\fa$-adic completion $\Lambda^{V(\fa)}M=\varprojlim M/\fa^n M$ of an $R$-module $M$. 
If the natural map $M\rightarrow {M}^{\wedge}_\fa$ is an isomorphism, then $M$ is called $\fa$-adically complete.
In addition, when $R$ is a local ring with maximal ideal $\fm$, we simply write $\widehat{M}$ for the $\fm$-adic completion of $M$.\medskip

We start with the following proposition.

\begin{proposition}\label{flatness}
Let $\fa$ be an ideal of $R$.
If $F$ is a flat $R$-module, then so is $F^\wedge_\fa$.
\end{proposition}

As stated in \cite[2.4]{Si}, this fact is known.
For the reader's convenience, we mention that this proposition follows from the two lemmas below.

\begin{lemma}\label{flat tensor completion}
Let $\fa$ be an ideal of $R$ and $F$ be a flat $R$-module.
We consider a short exact sequence of finitely generated $R$-modules
$$\begin{CD}
0@>>>&L@>>>&M@>>>&N@>>>0.
\end{CD}$$
Then 
$$\begin{CD} 
0@>>>&(F\otimes_RL )^\wedge_\fa@>>>&(F\otimes_RM)^\wedge_\fa @>>>&(F\otimes_R N)^\wedge_\fa @>>>0
\end{CD}$$
is exact.
\end{lemma}

\begin{lemma}\label{tensor representation}
Let $\fa$ and $F$ be as above. 
Then we have a natural isomorphism 
$$(F\otimes_RM)^{\wedge}_\fa\cong F^{\wedge}_\fa\otimes_{R}M$$
for any finitely generated $R$-module $M$.
\end{lemma}\medskip

Using Artin-Rees Lemma and \cite[I; \S 2; Proposition 6]{Bourbaki}, we can prove Lemma \ref{flat tensor completion}, from which we obtain Lemma \ref{tensor representation}.
Furthermore, Lemma \ref{flat tensor completion} and Lemma \ref{tensor representation} imply that $F^{\wedge}_\fa\otimes_{R}(-)$ is an exact functor from the category of finitely generated $R$-modules to $\Mod R$. Therefore Proposition \ref{flatness} holds.\qed\bigskip

It is also possible to show that $F^{\wedge}_\fa$ is flat over $R^\wedge_\fa$ by the same argument as above.\medskip

If $R$ is a local ring with maximal ideal $\fm$, then $\fm$-adically complete flat $R$-modules are characterized as follows:

\begin{lemma}\label{characterization complete flat}
Let $(R,\fm, k)$ be a local ring and $F$ a flat $R$-module. 
Set $B=\dim_{k} F/\fm F$.
Then there is an isomorphism
$$\widehat{F}\cong \widehat{\bigoplus_{B}R},$$
where $\bigoplus_{B}R$ is the direct sum of $B$-copies of $R$.
\end{lemma}\medskip

This lemma is proved in \cite[II; Proposition 2.4.3.1]{RG}. See also \cite[Lemma 6.7.4]{EJ}.\\

As in the introduction, we denote by $\cD=D(\Mod R)$ the derived category of all complexes of $R$-modules. 
We write complexes $X$ cohomologically;
$$X=(\cdots \rightarrow X^{i-1}\rightarrow X^i\rightarrow X^{i+1}\rightarrow\cdots).$$

For a complex $P$ (resp. $F$) of $R$-modules, 
we say that $P$ (resp. $F$) is $K$-projective (resp. $K$-flat) if $\Hom_R(P,-)$ (resp. $(-)\otimes_RF$) preserves acyclicity of complexes, where a complex is called acyclic if all its cohomology modules are zero.

Let $\fa$ be an ideal of $R$ and $X\in \cD$. If $P$ is a $K$-projective resolution of $X$, 
then we have $\LLambda^{V(\fa)}X\cong\Lambda^{V(\fa)}P$. 
Moreover, $\LLambda^{V(\fa)}X$ is also isomorphic to $\Lambda^{V(\fa)}F$ if $F$ is a $K$-flat resolution of $X$.
In addition, it is known that the following proposition holds. 

\begin{proposition}\label{not necessarily K-flat}
Let $\fa$ be an ideal of $R$ and $X$ be a complex of flat $R$-modules. Then $\LLambda^{V(\fa)}X$ is isomorphic to $\Lambda^{V(\fa)}X$.
\end{proposition}

To prove this proposition, we remark that there is an integer $n\geq 0$ such that $H^i(\LLambda^{V(\fa)}M)=0$ for all $i>n$ and all $R$-modules $M$, see \cite[Theorem 1.9]{GM} or \cite[p.\ 15]{AJL}.
Using this fact, we can show that $\Lambda^{V(\fa)}$ preserves acyclicity of complexes of flat $R$-modules.
Then it is straightforward to see that $\LLambda^{V(\fa)}X$ is isomorphic to $\Lambda^{V(\fa)}X$.
\qed\bigskip

Let $W$ be any subset of $\Spec R$.
Recall that $\gamma_W$ denotes a right adjoint to the inclusion functor $i_W:\cL_W\hookrightarrow \cD$, and $\lambda^W$ denotes a left adjoint to the inclusion functor $j^W:\cC^W\hookrightarrow \cD$. 
Moreover, $\gamma_W$ and $\lambda^W$ are identified with $i_W \gamma_W$ and $j^W\lambda^W$ respectively.
We write $\varepsilon_{W}:\gamma_{W}\to \id_{\cD}$ and $\eta^W:\id_\cD\to \lambda^W$ for the natural morphisms induced by the adjointness of $(i_{W},\gamma_{W})$ and $(\lambda^W, j^W)$ respectively.

Note that $\lambda^{W}\eta^W$ (resp. $\gamma_W\varepsilon_W$) is invertible, and the equality $\lambda^{W}\eta^W=\eta^W\lambda^{W}$ (resp. $\gamma_W\varepsilon_W=\varepsilon_W\gamma_W$) holds, {\it i.e.}, $\lambda^{W}$ (resp. $\gamma_W$) is a localization (resp. colocalization) functor on $\cD$. See \cite{K} for more details.
In this paper, we call $\lambda^W$ the localization functor with cosupport in $W$.\bigskip

Using (\ref{star}), we restate \cite[Lemma 2.1]{NY} as follows.

\begin{lemma} \label{basic lemma}
Let  $W$  be a subset of  $\Spec R$.
For any $X\in \cD$, there is a triangle of the following form;
$$
\begin{CD}
\gamma_{W^c} X @>\varepsilon_{W^c}X>> X @>\eta^WX>> \lambda^W X @>>> \gamma _{W^c} X[1].
\end{CD}
$$
Furthermore, if 
$$
\begin{CD}
X' @>>> X @>>> X''@>>> X'[1]
\end{CD}
$$
is a triangle with $X' \in{}^\perp\cC^{W}=\cL_{W^c}$ and  $X'' \in \cC^W=\cL_{W^c}^\perp$, then there exist unique isomorphisms $a:\gamma _{W^c} X \to X'$ and  $b:\lambda ^W X \to X''$ such that the following diagram is commutative:
$$
\begin{CD}
\gamma_{W^c} X @>\varepsilon_{W^c}X>> X @>\eta^WX>> \lambda^W X @>>> \gamma _{W^c} X[1]\\
@VVaV  @| @VVbV@VVa[1]V \\
X' @>>> X @>>> X''@>>> X'[1]
\end{CD}
$$
\end{lemma}\vspace{4mm}

\begin{remark}\label{characterization perp}
(i) Let $X\in \cD$ and $W$ be a subset of $\Spec R$.
By Lemma \ref{basic lemma}, $X$ belongs to ${}^\perp\cC^{W}=\cL_{W^c}$ if and only if $\lambda^{W}X=0$. 
This is equivalent to saying that $\lambda^{\{\fp\}}X= 0$ for all $\fp\in W$, since  ${}^\perp\cC^{W}=\cL_{W^c}=\bigcap_{\fp\in W}\cL_{\{\fp\}^c}=\bigcap_{\fp\in W}{}^\perp\cC^{\{\fp\}}$.

(ii) Let $W_0$ and $W$ be subsets of $\Spec R$ with $W_0\subseteq W$.
It follows from the uniqueness of adjoint functors that  
$$\lambda^{W_0}\lambda^{W}\cong \lambda^{W_0}\cong\lambda^{W}\lambda^{W_0},$$
see also \cite[Remark 3.7 (i)]{NY}.
\end{remark}

Now we give a typical example of localization functors.
Let $S$ be a multiplicatively closed subset $S$ of $R$, and set $U_S=\Set{\fp\in \Spec R | \fp\cap S=\emptyset}$.
It is known that the localization functor $\lambda^{U_S}$ with cosupport in $U_S$ is nothing but $(-)\otimes_RS^{-1}R$.
For the reader's convenience, we give a proof of this fact.
Let $X\in \cD$. It is clear that 
$\cosupp X\otimes_RS^{-1}R\subseteq U_S$, or equivalently, $X\otimes_RS^{-1}R\in \cC^{U_S}$.
Moreover, embedding the natural morphism $X\to X\otimes_RS^{-1}R$ into a triangle
$$\begin{CD} 
C@>>> X@>>> X\otimes_RS^{-1}R@>>>C[1],
\end{CD}$$
we have $C\otimes_RS^{-1}R=0$. 
This yields an inclusion relation $\supp C\subseteq (U_S)^c$. 
Hence it holds that $C\in \cL_{(U_S)^c}$.
Since we have shown that $C\in \cL_{(U_S)^c}$ and $X\otimes_RS^{-1}R\in \cC^{U_S}$, it follows from Lemma \ref{basic lemma} that $\lambda^{U_S}X\cong X\otimes_RS^{-1}R$.
Therefore we obtain the isomorphism
\begin{align}\lambda^{U_S}\cong (-)\otimes_RS^{-1}R.
\label{usual localization}
\end{align}

For $\fp \in \Spec R$, we write  $U(\fp)=\Set{\fq\in \Spec R| \fq\subseteq \fp}$. If $S=R\backslash \fp$, then  $U(\fp)$ is equal to $U_S$, so that $\lambda^{U(\fp)}\cong (-)\otimes_RR_\fp$ by (\ref{usual localization}).
We remark that $\lambda^{U(\fp)}=\lambda_{U(\fp)^c}$ is written as $L_{Z(\fp)}$ in \cite{BIK}, where $Z(\fp)=U(\fp)^c$.\medskip

There is another important example of localization functors.
Let $\fa$ be an ideal of $R$.
It was proved by \cite{GM} and \cite{AJL} that $\LLambda^{V(\fa)}:\cD\to \cD$ is a right adjoint to $\RGamma_{V(\fa)}:\cD\to \cD$. 
In \cite[Proposition 5.1]{NY}, using the adjointness property of $(\RGamma_{V(\fa)}, \LLambda^{V(\fa)})$, we proved that $\lambda^{V(\fa)}=\lambda_{V(\fa)^c}$ coincides with $\LLambda^{V(\fa)}$.
Hence there is an isomorphism
\begin{align}\lambda^{V(\fa)} \cong\LLambda^{V(\fa)}.
\label{derived local homology}
\end{align}
The functor $H_i^{\fa}(-)=H^{-i}(\LLambda^{V(\fa)}(-))$ is called the $i$th local homology functor with respect to $\fa$. \medskip

A subset $W$ of $\Spec R$ is said to be specialization-closed (resp. generalization-closed) provided that the following condition holds; if $\fp\in W$ and $\fq\in \Spec R$ with $\fp\subseteq \fq$ (resp. $\fp\supseteq \fq$), then $\fq\in W$.
 
If $V$ is a specialization-closed subset, then
we have 
\begin{align}
\gamma_{V} \cong \RGamma_{V},
\label{derived local cohomology}
\end{align}
see \cite[Appendix 3.5]{L}.\medskip


\section{Auxiliary results on localization functors}
\label{3}

In this section, we give several results to compute  localization functors $\lambda^W$ with cosupports in  arbitrary subsets $W$ of $\Spec R$.\medskip

We first give the following lemma.
\begin{lemma}\label{localizing and colocalizing}
Let $V$ be a specialization-closed subset of $\Spec R$.
Then we have the following equalities;
$$
{}^\perp\cC^{V}=\cL_{V^c}=\cL_{V}^\perp=\cC^{V^c}.
$$
\end{lemma}\medskip

\begin{proof}
This follows from \cite[Lemma 4.3]{NY} and (\ref{star}).
\end{proof}

Let $W$ be a subset of $\Spec R$.
We denote by $\overline{W}^s$ the specialization closure of $W$, which is the smallest specialization-closed subset of $\Spec R$ containing $W$.
Moreover, for a subset $W_0$ of $W$, we say that $W_0$ is specialization-closed in $W$ if $V(\fp)\cap W\subseteq W_0$ for any $\fp\in W_0$ (cf. \cite[Definition 3.10]{NY}). 
This is equivalent to saying that $\overline{W_0}^s\cap W=W_0$.

\begin{corollary}\label{perp cosupport ver}
Let $W_0 \subseteq W\subseteq \Spec R$ be sets. 
Suppose that $W_0$ is specialization-closed in $W$.
Setting $W_1 = W \ \backslash \ W_0$, we have 
$\cC^{W_1}\subseteq {}^\perp \cC^{W_0}$.
\end{corollary}

\begin{proof}
Notice that $W_1\subseteq (\overline{W_0}^s)^c$.
Furthermore, we have  ${}^\perp\cC^{\overline{W_0}^s}=\cC^{(\overline{W_0}^s)^c}$ by Lemma \ref{localizing and colocalizing}.
Hence it holds that   
$\cC^{W_1}\subseteq \cC^{(\overline{W_0}^s)^c}={}^\perp\cC^{\overline{W_0}^s}\subseteq {}^\perp\cC^{W_0}.
$
\end{proof}

\begin{remark}\label{adjoint pair specialization-closed}
For an ideal $\fa$ of $R$, $\lambda^{V(\fa)}$ is a right adjoint to $\gamma_{V(\fa)}$ by (\ref{derived local homology}) and (\ref{derived local cohomology}). 
More generally, it is known that for any specialization-closed subset $V$, $\lambda^V:\cD\to \cD$ is a right adjoint to $\gamma_V:\cD\to \cD$.
We now prove this fact, which will be used in the next proposition.
Let $X,Y\in \cD$, and consider the following triangles;
$$\begin{CD}
\gamma_{V}X@>>> X@>>> \lambda^{V^c}X@>>>\gamma_{V}X[1],\\
\gamma_{V^c}Y@>>> Y@>>> \lambda^{V}Y@>>>\gamma_{V^c}Y[1].\\
\end{CD}$$
Since $\lambda^{V^c}X\in \cC^{V^c}={}^\perp\cC^V$ by Lemma \ref{localizing and colocalizing}, applying $\Hom_{\cD}(-,\lambda^{V}Y)$  to the first triangle, we have
$\Hom_{\cD}(\gamma_{V}X,\lambda^{V}Y)\cong \Hom_{\cD}(X,\lambda^{V}Y)$.
Moreover, Lemma \ref{localizing and colocalizing} implies that
$\gamma_{V^c}Y\in \cL_{V^c}=\cL_{V}^\perp$. Hence,
applying $\Hom_{\cD}(\gamma_{V}X,-)$ to the second triangle, we have  
$\Hom_{\cD}(\gamma_{V}X,Y)\cong \Hom_{\cD}(\gamma_{V}X,\lambda^{V}Y)$.
Thus there is a natural isomorphism
$\Hom_{\cD}(\gamma_{V}X,Y)\cong \Hom_{\cD}(X,\lambda^{V}Y)$, so that $(\gamma_V, \lambda^V)$ is an adjoint pair.
See also \cite[Remark 5.2]{NY}.
\end{remark}

\begin{proposition}\label{general lambda}
Let $V$ and $U$ be arbitrary subsets of $\Spec R$.
Suppose that one of the following conditions holds:
\begin{itemize}
\item[\rm (1)] $V$ is specialization-closed;
\item[\rm (2)] $U$ is generalization-closed.
\end{itemize}
Then we have an isomorphism 
$$\lambda^V\lambda^U\cong \lambda^{V\cap U}.$$ 
\end{proposition}

\begin{proof}
Let $X\in \cD$ and $Y\in \cC^{V\cap U}=\cC^{V}\cap \cC^U$.
Then there are natural isomorphisms
$$\Hom_{\cD}(\lambda^{V}\lambda^{U}X,Y)\cong \Hom_{\cD}(\lambda^{U}X,Y)\cong \Hom_{\cD}(X,Y).$$
Recall that $\lambda^{V\cap U}$ is a left adjoint to the inclusion functor $\cC^{V\cap U}\hookrightarrow \cD$.
Hence, by the uniqueness of adjoint functors, we only have to verify that $\lambda^{V}\lambda^{U} X\in \cC^{V\cap U}$. 
Since $\lambda^{V}\lambda^UX\in \cC^{V}$, it remains to show that $\lambda^{V}\lambda^UX\in \cC^U$.

Case (1): 
Let $\fp\in U^c$.
Since $\supp \gamma_{V}\kappa(\fp)\subseteq \{\fp\}$, it follows from (\ref{star}) that $\gamma_{V}\kappa(\fp)\in \cL_{U^c}={}^\perp\cC^{U}$.
Thus, by the adjointness of $(\gamma_V, \lambda^V)$, we have 
$$\RHom_R(\kappa(\fp),\lambda^{V} \lambda^UX)\cong \RHom_R(\gamma_{V}\kappa(\fp), \lambda^UX)=0.$$
This implies that $\cosupp \lambda^{V}\lambda^UX\subseteq U$, {\it i.e.}, $\lambda^{V}\lambda^UX\in \cC^{U}$.

Case (2): Since $U^c$ is specialization-closed, the case (1) yields an isomorphism $\lambda^{U^c}\lambda^{V}\cong \lambda^{U^c\cap V}$.
Furthermore, setting $W=(U^c\cap V) \cup U$, we see that $U^c\cap V$ is specialization-closed in $W$, and $W\backslash(U^c\cap V)=U$. 
Hence we have $\lambda^{U^c}(\lambda^{V}\lambda^UX)\cong \lambda^{U^c\cap V}\lambda^UX=0$, by Corollary \ref{perp cosupport ver}. 
It then follows from Lemma \ref{localizing and colocalizing} that $\lambda^{V}\lambda^UX\in {}^\perp \cC^{U^c}=\cC^U$.
\end{proof}

\begin{remark}\label{general remark lambda}
For arbitrary subsets $W_0$ and $W_1$ of $\Spec R$, 
Remark \ref{characterization perp} (ii) and Proposition \ref{general lambda} yield the following isomorphisms;
$$\lambda^{W_0}\lambda^{W_1}\cong \lambda^{W_0}\lambda^{\overline{W_0}^s} \lambda^{W_1}\cong \lambda^{W_0}\lambda^{\overline{W_0}^s\cap W_1},$$
$$\lambda^{W_0}\lambda^{W_1}\cong \lambda^{W_0}\lambda^{\overline{W_1}^g} \lambda^{W_1}\cong \lambda^{W_0\cap \overline{W_1}^g}\lambda^{W_1}.$$
\end{remark}\vspace{2mm}

The next result is a corollary of (\ref{usual localization}), (\ref{derived local homology}) and Proposition \ref{general lambda}.

\begin{corollary}\label{lambda 1}
Let $S$ be a multiplicatively closed subset of $R$ and $\fa$ be an ideal of $R$.
We set  
$W=V(\fa)\cap U_S$.
Then we have
$$
\lambda^W \cong \LLambda^{V(\fa)}( -\otimes_RS^{-1}R).
$$
\end{corollary}\vspace{4mm}
 
Since $V(\fp)\cap U(\fp)=\{\fp\}$ for $\fp\in \Spec R$, 
as a special case of this corollary, we have the following result.
 
\begin{corollary}\label{lambda p}
Let $\fp$ be a prime ideal of $R$. 
Then we have
$$\lambda^{\{\fp\}}
\cong \LLambda^{V(\fp)}( -\otimes_RR_\fp).
$$
\end{corollary}\vspace{4mm}

The next lemma follows from this corollary and Lemma \ref{characterization complete flat}.

\begin{lemma}\label{lambda p 2}
Let $\fp$ be a prime ideal of $R$ and $F$ be a flat $R$-module.
Then $\lambda^{\{\fp\}}F$ is isomorphic to $(\bigoplus_BR_\fp)^\wedge_\fp$,
where $\bigoplus_BR_\fp$ is the direct sum of $B$-copies of $R_\fp$ and $B=\dim_{\kappa(\fp)} F\otimes_R\kappa(\fp)$.
\end{lemma}

\begin{remark}\label{not commutative}
If $W_1$ and $W_2$ are both specialization-closed or both generalization-closed, then Proposition \ref{general lambda} implies that $\lambda^{W_1}\lambda^{W_2}\cong\lambda^{W_2}\lambda^{W_1}$. However, in general, $\lambda^{W_1}$ and $\lambda^{W_2}$ need not commute.
For example, let $\fp, \fq\in \Spec R$ with $\fp\subsetneq \fq$.
Then $(\lambda^{\{\fp\}}R)\otimes_R\kappa(\fq)=\widehat{R_{\fp}}\otimes_R\kappa(\fq)=0$ and $(\lambda^{\{\fq\}}R)\otimes_R\kappa(\fp)=\widehat{R_{\fq}}\otimes_R\kappa(\fp)\neq 0$.
Then we see from Lemma \ref{lambda p 2}
that $\lambda^{\{\fq\}}\lambda^{\{\fp\}}R=0$ and $\lambda^{\{\fp\}}\lambda^{\{\fq\}}R\neq 0$.
\end{remark}\vspace{1mm}

Compare this remark with \cite[Example 3.5]{BIK}.
See also \cite[Remark 3.7 (ii)]{NY}.\medskip

Let $\fp$ be a prime ideal which is not maximal.
Then $\lambda^{\{\fp\}}$ is distinct from $\varLambda^{\fp}=\LLambda^{V(\fp)}\RHom_R(R_\fp,-)$, which is introduced in \cite{BIK3}.
To see this, let  $\fq$ be a prime ideal with $\fp\subsetneq \fq$. Then it holds that $\cosupp \widehat{R_\fq}=\{\fq\}\subseteq U(\fp)^c$.
Hence $\widehat{R_\fq}$ belongs to $\cC^{U(\fp)^c}$.
Then we have $\RHom_R(R_\fp,\widehat{R_\fq})=0$ since $R_\fp\in \cL_{U(\fp)}={}^\perp\cC^{U(\fp)^c}$ by (\ref{star}).
This implies that $\varLambda^{\fp}\widehat{R_\fq}=\LLambda^{V(\fp)}\RHom_R(R_\fp,\widehat{R_\fq})=0$, while $\lambda^{\{\fp\}}\widehat{R_\fq}\cong\lambda^{\{\fp\}}\lambda^{\{\fq\}}R\neq 0$ by Remark \ref{not commutative}.

Let $X\in \cD$, and write $\varGamma_\fp=\RGamma_{V(\fp)}(-\otimes_RR_\fp)$ (cf. \cite{BIK}). 
Recall that $\fp\in \supp X$ (resp. $\fp\in \cosupp X$) if and only if $\varGamma_\fp X\neq 0$ (resp. $\varLambda^{\fp}X\neq 0$), see \cite[Theorem 2.1, Theorem 4.1]{FI} and \cite[\S 4]{BIK3}.
In contrast, $\fp\in \cosupp X$ (resp. $\fp\in \supp X$) if and only if $\gamma_{\{\fp\}}X\neq 0$ (resp. $\lambda^{\{\fp\}} X\neq 0$), by Lemma \ref{basic lemma}. 
Here $\gamma_{\{\fp\}}\cong \RGamma_{V(\fp)}\RHom_R(R_\fp,-)$ by \cite[Corollary 3.3]{NY}.
See also \cite[Proposition 3.6, Proposition 4.4]{SW}.
\medskip

Let $W$ be a subset of $\Spec R$. We denote by $\dim W$ the supremum of lengths of chains of distinct prime ideals in $W$ (cf. \cite[Definition 3.6]{NY}).

\begin{theorem}\label{lambda dim W=0}
Let $W$ be a subset of $\Spec R$. We assume that $\dim W=0$. Then there are isomorphisms
$$\lambda^W\cong \prod_{\fp \in W}\lambda^{\{\fp\}}\cong \prod_{\fp \in W}\LLambda^{V(\fp)}( -\otimes_RR_\fp).$$
\end{theorem}

\begin{proof}
Let $X\in \cD$, and consider the natural morphisms $\eta^{\{\fp\}}X:X\to \lambda^{\{\fp\}}X$ for $\fp\in W$.
Take the product of the morphisms, and we obtain a morphism $f:X\to \prod_{\fp\in W}\lambda^{\{\fp\}}X$. 
Embed $f$ into a triangle
$$
\begin{CD}
C@>>>X@>f>>\disp{\prod_{\fp\in W}\lambda^{\{\fp\}}X}@>>>C[1].
\end{CD}
$$
Note that $\prod_{\fp\in W}\lambda^{\{\fp\}}X\in \cC^W$.
We have to prove that $C\in {}^\perp\cC^W$.
For this purpose, take any prime ideal $\fq\in W$. 
Then $\{\fq\}$ is specialization-closed in $W$, because $\dim W=0$. 
Hence we have $\prod_{\fp\in W\backslash\{\fq\}}\lambda^{\{\fp\}}X\in \cC^{W\backslash \{\fq\}}\subseteq {}^\perp\cC^{\{\fq\}}$, by Corollary \ref{perp cosupport ver}.
Thus an isomorphism $\lambda^{\{\fq\}}(\prod_{\fp\in W}\lambda^{\{\fp\}}X)\cong \lambda^{\{\fq\}}X$ holds.
Then it is seen from the triangle above that $\lambda^{\{\fq\}}C=0$ for all $\fq\in W$, so that $C\in {}^\perp\cC^W$, see Remark \ref{characterization perp} (i). 
Therefore Lemma \ref{basic lemma} yields $\lambda^WX\cong \prod_{\fp\in W}\lambda^{\{\fp\}}X$. 
The second isomorphism in the theorem follows from  Corollary \ref{lambda p}.
\end{proof}

\begin{example}\label{example lambda infinite}
Let $W$ be a subset of $\Spec R$ such that $W$ is an infinite set with $\dim W=0$.
Let $X^{\{\fp\}}$ be a complex with $\cosupp X^{\{\fp\}}=\{\fp\}$ for each $\fp\in W$.
We take $\fp\in W$.
Since $\dim W=0$, it holds that $X^{\{\fq\}}\in \cC^{V(\fp)^c}$ for any $\fq\in W\backslash\{\fp\}$. 
Furthermore, Lemma \ref{localizing and colocalizing} implies that $\cC^{V(\fp)^c}$ is equal to ${}^\perp\cC^{V(\fp)}$, which is closed under arbitrary direct sums.
Thus it holds that $\bigoplus_{\fq\in W\backslash \{\fp\}}X^{\{\fq\}}\in \cC^{V(\fp)^c}={}^\perp\cC^{V(\fp)}\subseteq {}^\perp\cC^{\{\fp\}}$.
Therefore, setting $Y=\bigoplus_{\fp\in W}X^{\{\fp\}}$, we have
$\lambda^{\{\fp\}}Y\cong X^{\{\fp\}}$.
It then follows from Theorem \ref{lambda dim W=0} that  
$$\lambda^WY\cong \prod_{\fp\in W}\lambda^{\{\fp\}}Y\cong \prod_{\fp\in W}X^{\{\fp\}}.$$
Under this identification, the natural morphism $Y\to \lambda^WY$ coincides with the canonical morphism $\bigoplus_{\fp\in W}X^{\{\fp\}}\to \prod_{\fp\in W}X^{\{\fp\}}$.
\end{example}

\begin{remark}
Let $W$, $X^{\{\fp\}}$ be as in Example \ref{example lambda infinite}, and suppose that each $X^{\{\fp\}}$ is an $R$-module. Then $\bigoplus_{\fp\in W}X^{\{\fp\}}$ is not in $\cC^W$, because the natural morphism $\bigoplus_{\fp\in W}X^{\{\fp\}}\to \lambda^W(\bigoplus_{\fp\in W}X^{\{\fp\}})$ is not an isomorphism.
Hence the cosupport of $\bigoplus_{\fp\in W}X^{\{\fp\}}$ properly contains $W$.
In particular, setting $X^{\{\fp\}}=\kappa(\fp)$, we have $W\subsetneq \cosupp \bigoplus_{\fp\in W}\kappa(\fp)$. Similarly, we can prove that $W\subsetneq \supp \prod_{\fp\in W}\kappa(\fp)$.
The first author noticed these facts through discussion with Srikanth Iyengar.

It is possible to give another type of examples, by which we also see that a colocalizing subcategory of $\cD$ is not necessarily closed under arbitrary direct sums.
Suppose that $(R,\fm)$ is a complete local ring with $\dim R\geq 1$.
Then we have $R\cong \widehat{R} \in \cC^{V(\fm)}$.
However the free module $\bigoplus_{\mathbb{N}} R$ is  never $\fm$-adically complete, so that $\bigoplus_{\mathbb{N}} R$ is not isomorphic to $\lambda^{V(\fm)}(\bigoplus_{\mathbb{N}}R)$.
Hence $\bigoplus_{\mathbb{N}} R$ is not in $\cC^{V(\fm)}$. 
\end{remark}

For a subset $W$ of $\Spec R$, $\overline{W}^g$ denotes the generalization closure of $W$, which is the smallest generalization-closed subset of $\Spec R$ containing $W$.
In addition, for a subset $W_1\subseteq W$, we say that $W_1$ is generalization-closed in $W$ if $W\cap U(\fp)\subseteq W_1$ for any $\fp\in W_1$. This is equivalent to saying that $W\cap \overline{W_1}^g=W_1$.

We extend Proposition \ref{general lambda} to the following corollary, which will be used in a main theorem of this section.

\begin{corollary}\label{composition 2}
Let $W_0$ and $W_1$ be arbitrary subsets of $\Spec R$.
Suppose that one of the following conditions hold:
\begin{itemize}
\item[\rm (1)] $W_0$ is specialization-closed in $W_0\cup W_1$;
\item[\rm (2)] $W_1$ is generalization-closed in $W_0\cup W_1$.
\end{itemize}
Then we have an isomorphism
$$\lambda^{W_0}\lambda^{W_1}\cong \lambda^{W_0\cap W_1}.$$
\end{corollary}

\begin{proof}
Set $W=W_0\cup W_1$.
By the assumption, we have $$\overline{W_0}^s\cap W=W_0\  \text{ or }\  W\cap \overline{W_1}^g=W_1.$$
Therefore, it holds that 
$$\overline{W_0}^s\cap W_1=W_0\cap W_1\  \text{ or }\ W_0\cap \overline{W_1}^g=W_0\cap W_1.$$
Hence this proposition follows from Remark \ref{general remark lambda} and Remark \ref{characterization perp} (ii). 
\end{proof}\vspace{1mm}

\begin{remark}\label{remark for main result 1}
(i) Let $W_0$ and $W$ be subsets of $\Spec R$ with $W_0\subseteq W$.
Under the isomorphism $\lambda^{W_0}\lambda^W\cong\lambda^{W_0}$ by Remark \ref{characterization perp} (ii), there is a morphism $\eta^{W_0}\lambda^{W}:\lambda^{W}\to \lambda^{W_0}$.

(ii) Let $W_0$ and $W_1$ be subsets of $\Spec R$.
Let $X\in \cD$.
Since $\eta^{W_1} :\id_{\cD}\to \lambda^{W_1}$ is a morphism of functors, there is a commutative diagram of the following form: 
$$\begin{CD}
X@>\eta^{W_0}X>>\lambda^{W_0}X\\
@VV\eta^{W_1}XV@VV\eta^{W_1}\lambda^{W_0}XV\\
\lambda^{W_1}X@>\lambda^{W_1}\eta^{W_0}X>>\lambda^{W_1}\lambda^{W_0}X
\end{CD}$$
\end{remark}\vspace{4mm}

Now we prove the following result, which is the main theorem of this section.

\begin{theorem}\label{main result 1}
Let $W$, $W_0$ and $W_1$  be subsets of $\Spec R$ with $W=W_0\cup W_1$.
Suppose that one of the following conditions holds:
\begin{itemize}
\item[\rm (1)] $W_0$ is specialization-closed in $W$;
\item[\rm (2)] $W_1$ is generalization-closed in $W$.
\end{itemize}
Then, for any $X\in \cD$, there is a triangle of the following form;
$$\begin{CD}
\lambda^{W} X@>f>>\lambda^{W_1}X\oplus\lambda^{W_0} X @>g>>  \lambda^{W_1}\lambda^{W_0} X @>>> \lambda^{W} X[1],
\end{CD}$$
where $f$ and $g$ are morphisms represented by the following matrices; 
$$f=\left(\begin{array}{c}
\eta^{W_1}\lambda^{W}X\\[7pt]  \eta^{W_0}\lambda^WX
\end{array}\right),\ \    
g=\left(\begin{array}{cc}\lambda^{W_1}\eta^{W_0}X\ \ \   
(-1)\cdot \eta^{W_1}\lambda^{W_0}X
\end{array}\right).$$
\end{theorem}\vspace{2mm}

\begin{proof}
We embed the morphism $g$ into a triangle
$$\begin{CD}
C@>a>>\lambda^{W_1}X\oplus\lambda^{W_0}X@>g>>\lambda^{W_1}\lambda^{W_0}X@>>>C[1].
\end{CD}$$
Notice that $C\in \cC^W$ since $\cC^{W_0}, \cC^{W_1}\subseteq \cC^W$.
By Remark \ref{remark for main result 1}, it is easily seen that $g\cdot f=0$. Thus there is a morphism $b: \lambda^WX\to C$ making the following diagram commutative:
\begin{align}
\begin{CD}
\lambda^WX@=\lambda^WX@>>>0@>>> \lambda^WX[1]\\
@VVbV@VVfV@VVV@VVb[1]V\\
C@>a>> \lambda^{W_1}X\oplus\lambda^{W_0} X @>g>>  \lambda^{W_1}\lambda^{W_0}X@>>> C[1]
\end{CD}\label{main thm diagram1}
\end{align}
We only have to show that $b$ is an isomorphism.
To do this, embedding the morphism $b$ into a triangle
\begin{align}
\begin{CD}
Z@>>>\lambda^WX@>b>>C@>>>Z[1],
\end{CD}\label{main thm diagram2}
\end{align}
we prove that $Z=0$.
Since $\lambda^WX, C\in \cC^W$, $Z$ belongs to  $\cC^W$. 
Hence it suffices to show that $Z\in {}^\perp\cC^W$.

First, we prove that $\lambda^{W_1}b$ is an isomorphism. 
We employ a similar argument to \cite[Theorem 7.5]{BIK}.
Consider the following sequence 
\begin{align}
\begin{CD}
\lambda^{W}X@>f>>\lambda^{W_1}X\oplus\lambda^{W_0}X@>g>>\lambda^{W_1}\lambda^{W_0}X,
\end{CD}\label{main thm diagram3}
\end{align}
and apply $\lambda^{W_1}$ to it.
Then we obtain a sequence which can be completed to a split triangle.
The triangle appears in the first row of the diagram below. 
Moreover, $\lambda^{W_1}$ sends the second row of the diagram (\ref{main thm diagram1}) to a split triangle, which appears in the second row of the the diagram below.
$$\begin{CD}
\lambda^{W_1}X@>\lambda^{W_1}f>>\lambda^{W_1}X\oplus\lambda^{W_1}\lambda^{W_0}X@>\lambda^{W_1}g>>\lambda^{W_1}\lambda^{W_0}X@>0>>\lambda^{W_1}X[1]\\
@VV\lambda^{W_1}bV@|@|@VV\lambda^{W_1}b[1]V\\
\lambda^{W_1}C@>\lambda^{W_1}a>>\lambda^{W_1}X\oplus\lambda^{W_1}\lambda^{W_0}X@>\lambda^{W_1}g>>\lambda^{W_1}\lambda^{W_0}X@>0>>\lambda^{W_1}C[1]
\end{CD}$$
Since this diagram is commutative, we conclude that $\lambda^{W_1} b$ is an isomorphism.

Next, we prove that $\lambda^{W_0}b$ is an isomorphism. 
Thanks to Corollary \ref{composition 2}, we are able to follow the same process as above.
In fact, the corollary implies that $\lambda^{W_0}\lambda^{W_1}\cong \lambda^{W_0\cap W_1}$.
Thus, applying $\lambda^{W_0}$ to the sequence (\ref{main thm diagram3}), we obtain a sequence which can be completed into a split triangle.
Furthermore, $\lambda^{W_0}$ sends the second row of the diagram (\ref{main thm diagram1}) to a split triangle.
Consequently we see that there is a morphism of triangles:
$$\begin{CD}
\lambda^{W_0}X@>\lambda^{W_0}f>>\lambda^{W _0\cap W_1}X\oplus \lambda^{W_0}X@>\lambda^{W_0}g>> \lambda^{W_0\cap W_1}X
@>0>>\lambda^{W_0}X[1]\\
@VV\lambda^{W_0}bV@|@|@VV\lambda^{W_0}b[1]V\\
\lambda^{W_0}C@>\lambda^{W_0}a>> \lambda^{W_0\cap W_1}X\oplus\lambda^{W_0} X @>\lambda^{W_0}g>>  \lambda^{W_0\cap W_1}X@>0>> \lambda^{W_0}C[1]
\end{CD}$$
Therefore $\lambda^{W_0}b$ is an isomorphism.

Since we have shown that $\lambda^{W_0}b$ and $\lambda^{W_1}b$ are isomorphisms, it follows from the triangle (\ref{main thm diagram2}) that $\lambda^{W_0}Z=\lambda^{W_1}Z=0$.
Thus we have $Z\in {}^\perp\cC^W$ by Remark \ref{characterization perp} (i).
\end{proof}

\begin{remark}\label{remark for the natural morphism}
Let $f$, $g$ and $a$ as above.
Let $h:X\to \lambda^{W_1}X\oplus\lambda^{W_0} X$ be a morphism induced by  
$\eta^{W_1}X$ and $\eta^{W_0}X$.
Then $g\cdot h=0$ by Remark \ref{remark for main result 1} (ii). Hence there is a morphism $b':X\to C$ such that the following diagram is commutative:
$$\begin{CD}
X@=X@>>>0@>>> X[1]\\
@VVb'V@VVhV@VVV@VVb'[1]V\\
C@>a>> \lambda^{W_1}X\oplus\lambda^{W_0} X @>g>>  \lambda^{W_1}\lambda^{W_0}X@>>> C[1].
\end{CD}$$
We can regard any morphism $b'$ making this diagram commutative as the natural morphism $\eta^{W}X$.
In fact, since $\lambda^{W}h=f$, applying $\lambda^{W}$ to this diagram, and setting $\lambda^{W}b'=b$, we obtain the diagram (\ref{main thm diagram1}).
Note that $b\cdot \eta^WX=b'$. Moreover, the above proof implies that $b:\lambda^{W}X\to C$ is an isomorphism. 
Thus we can identify $b'$ with $\eta^WX$ under the isomorphism $b$.
\end{remark}

We give some examples of Theorem \ref{main result 1}.

\begin{example}\label{examples Mayer-Vietoris}
(1) Let $x$ be an element of $R$.
Recall that $\lambda^{V(x)}\cong \LLambda^{V(x)}$ by (\ref{derived local homology}).
We put $S=\Set{x^n |n\geq 0}$.
Since $V(x)^c=U_S$, it holds that $\lambda^{V(x)^c}=\lambda^{U_S}\cong (-)\otimes_RR_x$ by (\ref{usual localization}). 
Set $W=\Spec R$, $W_0=V(x)$ and $W_1=V(x)^c$.
Then the theorem yields the following triangle 
$$\begin{CD}
R@>>>R_x\oplus R^\wedge_{(x)}@>>>(R^\wedge_{(x)})_x@>>>R[1].
\end{CD}$$
(2) Suppose that $(R,\fm)$ is a local ring with $\fp\in \Spec R$ and having $\dim R/\fp=1$.
Setting $W=V(\fp)$, $W_0=V(\fm)$ and $W_1=\{\fp\}$,  
we see from the theorem and Corollary \ref{lambda p} that there is a short exact sequence
$$\begin{CD}
0@>>>R^\wedge_{\fp}@>>>\widehat{R_{\fp}}\oplus \widehat{R}@>>>\widehat{(\widehat{R})_\fp}@>>>0.
\end{CD}$$
Actually, this gives a pure-injective resolution of $R^\wedge_{\fp}$, see Section 9.
Moreover, if $R$ is a 1-dimensional local domain with quotient field $Q$, then this short exact sequence is of the form
$$\begin{CD}
0@>>>R@>>>Q \oplus \widehat{R}@>>> \widehat{R}\otimes_RQ@>>> 0.
\end{CD}
$$
\end{example}\vspace{5mm}

By similar arguments to Proposition \ref{general lambda} and Corollary \ref{composition 2}, one can prove the following proposition, which is a generalized form of \cite[Proposition 3.1]{NY}.

\begin{proposition}\label{general gamma}
Let $W_0$ and $W_1$ be arbitrary subsets of $\Spec R$.
Suppose that one of the following conditions hold:
\begin{itemize}
\item[\rm (1)] $W_0$ is specialization-closed in $W_0\cup W_1$;
\item[\rm (2)] $W_1$ is generalization-closed in $W_0\cup W_1$.
\end{itemize}
Then we have an isomorphism
$$\gamma_{W_0}\gamma_{W_1}\cong \gamma_{W_0\cap W_1}.$$ 
\end{proposition}\vspace{3mm}

As with Theorem \ref{main result 1}, it is possible to prove the following theorem, in which we implicitly use the fact that $\gamma_{W_0}\gamma_W\cong \gamma_{W_0}$ if $W_0\subseteq W$ (cf. \cite[Remark 3.7 (i)]{NY}).

\begin{theorem}\label{gamma pushout}
Let $W$, $W_0$ and $W_1$  be subsets of $\Spec R$ with $W=W_0\cup W_1$.
Suppose that one of the following conditions holds:
\begin{itemize}
\item[\rm (1)] $W_0$ is specialization-closed in $W$;
\item[\rm (2)] $W_1$ is generalization-closed in $W$.
\end{itemize}
Then, for any $X\in \cD$, there is a triangle of the following form;
$$\begin{CD}
\gamma_{W_1}\gamma_{W_0} X@>f>>\gamma_{W_1}X\oplus\gamma_{W_0} X @>g>>  \gamma_{W} X @>>> \gamma_{W_1}\gamma_{W_0} X[1],
\end{CD}$$
where $f$ and $g$ are morphisms represented by the following matrices;
$$
f=\left(\begin{array}{c}
\gamma_{W_1}\varepsilon_{W_0}X\\[7pt] (-1) \cdot \varepsilon_{W_1}\gamma_{W_0}X
\end{array}\right),\ \ 
g=\left(\begin{array}{cc}\varepsilon_{W_1}\gamma_{W}X\ \ \ 
\varepsilon_{W_0}\gamma_WX
\end{array}\right)
.$$
\end{theorem}\vspace{3mm}

\begin{remark}
As long as we work on the derived category $\cD$, 
Theorem \ref{main result 1} and Theorem \ref{gamma pushout} generalize Mayer-Vietoris triangles in the sense of Benson, Iyengar and Krause \cite[Theorem 7.5]{BIK}, in which $\gamma_V$ and $\lambda_V$ are written as $\varGamma_V$ and $L_V$ respectively for a specialization-closed subset $V$ of $\Spec R$.
\end{remark}


\section{Projective dimension of flat modules}

As an application of results in Section 3, we give a simpler proof of a classical theorem due to Gruson and Raynaud.

\begin{theorem}[{\cite[II; Corollary 3.2.7]{RG}}]\label{RG}
Let $F$ be a flat $R$-module. Then the projective dimension of F is at most $\dim R$.
\end{theorem}

We start by showing the following lemma.

\begin{lemma}\label{flat RHom lemma}
Let $F$ be a flat $R$-module and $\fp$ be a prime ideal of $R$.
Suppose that $X\in \cC^{\{\fp\}}$. 
Then there is an isomorphism
$$\RHom_R(F,X)\cong \prod_{B}X,$$
where $B=\dim_{\kappa(\fp)}F\otimes_R\kappa(\fp)$.
\end{lemma}

\begin{proof}
Since $\lambda^{\{\fp\}}:\cD\to \cC^{\{\fp\}}$ is a left adjoint to the inclusion functor $ \cC^{\{\fp\}}\hookrightarrow\cD$, we have $\RHom_R(F,X)\cong\RHom_R(\lambda^{\{\fp\}}F,X)$.
Moreover it follows from Lemma \ref{lambda p 2} that 
$\lambda^{\{\fp\}}F\cong (\bigoplus_{B}R_\fp)^\wedge_\fp\cong \lambda^{\{\fp\}}(\bigoplus_{B}R)$, where $B=\dim_{\kappa(\fp)}F\otimes_R\kappa(\fp)$.
Therefore
we obtain isomorphisms 
$\RHom_R(F,X)\cong \RHom_R(\lambda^{\{\fp\}}(\bigoplus_{B}R), X)\cong \RHom_R(\bigoplus_{B}R, X)\cong \prod_{B}X.$
\end{proof}\vspace{2mm}
Let $a, b\in \Z\cup \{\pm\infty\}$ with $a \leq b$.
We write $\cD^{{[a,b]}}$ for the full subcategory of $\cD$ consisting of all complexes $X$ of $R$-modules such that $H^i(X)=0$ for $i\notin[a, b]$ (cf. \cite[Notation 13.1.11]{KS}).
For a subset $W$ of $\Spec R$, $\max W$ denotes the set of prime ideals $\fp\in W$ which are maximal with respect to inclusion in $W$.

\begin{proposition}\label{flat key corollary}
Let $F$ be a flat $R$-module and  $X\in \cD^{[-\infty ,0]}$.
Suppose that  $W$ is a subset of $\Spec R$ such that  $n=\dim W$ is finite.
Then we have $\Ext_R^i(F, \lambda^{W}X)=0$ for $i>n$.
\end{proposition}

\begin{proof}
We use induction on $n$.
First, we suppose that $n=0$. 
It then holds that $\lambda^{W}X\cong \prod_{\fp\in W}\lambda^{\{\fp\}}X\cong \prod_{\fp\in W}\LLambda^{V(\fp)}X_\fp\in \cD^{[-\infty ,0]}$, by Theorem \ref{lambda dim W=0}. 
Hence, noting that $\RHom_R(F, \lambda^WX)\cong \prod_{\fp\in W}\RHom_R(F,  \lambda^{\{\fp\}}X)$, we have $\Ext^i_R(F,\lambda^{W}X)=0$ for $i>0$, by Lemma \ref{flat RHom lemma}.

Next, we suppose $n>0$.
Set $W_0=\max W$ and $W_1=W\backslash W_0$. 
By Theorem \ref{main result 1}, there is a triangle
$$\begin{CD}
\lambda^{W} X@>>>\lambda^{W_1}X\oplus\lambda^{W_0} X @>>>  \lambda^{W_1}\lambda^{W_0} X @>>> \lambda^{W} X[1].
\end{CD}$$
Note that $\dim W_0=0$ and $\dim W_1=n-1$.
By the argument above, it holds that    
$\Ext^i_R(F,\lambda^{W_0}X)=0$
 for $i>0$. 
Furthermore, since $X,\lambda^{W_0}X\in \cD^{[-\infty ,0]}$, we have  
$\Ext_R^i(F,\lambda^{W_1}X)=\Ext^i_R(F,\lambda^{W_1}\lambda^{W_0}X)=0
$ for $i>n-1$, by the inductive hypothesis.
Hence it is seen from the triangle that $\Ext^i_R(F,\lambda^{W}X)=0$ for $i>n$.
\end{proof}

\noindent {\it Proof of Theorem \ref{RG}.}
We may assume that $d=\dim R$ is finite.
Let $M$ be any $R$-module.
We only have to show that $\Ext_R^i(F,M)=0$ for $i>d$.
Setting $W=\Spec R$, we have $\dim W=d$ and $M\cong \lambda^{W}M$.
It then follows from Proposition \ref{flat key corollary} that 
$\Ext^i_R(F, M)\cong \Ext^i_R(F,  \lambda^WM)=0$ for $i> d$.
\qed




\section{Cotorsion flat modules and cosupport}

In this section, we summarize some basic facts about cotorsion flat $R$-modules.\medskip

Recall that an $R$-module $M$ is called cotorsion if $\Ext^1_R(F,M)=0$ for any flat $R$-module $F$.
This is equivalent to saying that $\Ext^i_R(F,M)=0$ for any flat $R$-module $F$ and any $i>0$.
Clearly, all injective $R$-modules are cotorsion.

A cotorsion flat $R$-module means an $R$-module which is cotorsion and flat. 
If $F$ is a flat $R$-module and $\fp\in \Spec R$, then Corollary \ref{lambda p} implies that $\lambda^{\{\fp\}}F$ is isomorphic to $\widehat{F_\fp}$, which is a  cotorsion flat $R$-module by Lemma \ref{flat RHom lemma} and Proposition \ref{flatness}.
Moreover, recall that $\widehat{F_\fp}$ is isomorphic to the $\fp$-adic completion of a free $R_\fp$-module by Lemma \ref{lambda p 2}. 

We remark that arbitrary direct products of flat $R$-modules are flat, since $R$ is Noetherian.
Hence, if $T_\fp$ is the $\fp$-adic completion of a free $R_\fp$ module for each $\fp\in \Spec R$, 
then $\prod_{\fp\in \Spec R} T_\fp$ is a cotorsion flat $R$-module.
Conversely, the following fact holds.

\begin{proposition}[{Enochs \cite{E1}}]\label{Enochs theorem}
Let $F$ be a cotorsion flat $R$-module. 
Then there is an isomorphism 
$$F\cong \prod_{\fp\in \Spec R} T_\fp,$$
where $T_\fp$ is the $\fp$-adic completion of a free $R_\fp$ module.
\end{proposition}

\begin{proof}
See \cite[Theorem]{E1} or \cite[Theorem 5.3.28]{EJ}.
\end{proof}\vspace{1mm}

Let $S$ be a multiplicatively closed subset of $R$ and $\fa$ be an ideal of $R$.
For a cotorsion flat $R$-module $F$, we have $\RHom_R(S^{-1}R, F)\cong \Hom_R(S^{-1}R, F)$ and 
$\LLambda^{V(\fa)}F\cong \Lambda^{V(\fa)}F$. 
Moreover, by Proposition \ref{Enochs theorem}, we may regard $F$ as an $R$-module of the form $\prod_{\fp\in \Spec R} T_\fp$.
Then it holds that
\begin{align}
\label{colocalization of cotorsion flat}
\RHom_R(S^{-1}R, \prod_{\fp\in \Spec R} T_\fp)\cong \Hom_R(S^{-1}R, \prod_{\fp\in \Spec R} T_\fp)\cong\prod_{\fp\in U_S} T_\fp.
\end{align}
This fact appears implicitly in \cite[\S 5.2]{X}.
Furthermore we have 
\begin{align}
\label{completion of cotorsion flat}
\LLambda^{V(\fa)}\prod_{\fp\in \Spec R} T_\fp\cong \Lambda^{V(\fa)}\prod_{\fp\in \Spec R} T_\fp\cong  \prod_{\fp\in V(\fa)} T_\fp.
\end{align}
One can show (\ref{colocalization of cotorsion flat}) and (\ref{completion of cotorsion flat}) by Lemma \ref{localizing and colocalizing} and (\ref{derived local homology}).
See also the recent paper \cite[Lemma 2.2]{PT} of Thompson.\medskip

Let $F$ be a cotorsion flat $R$-module with $\cosupp F\subseteq W$ for a subset $W$ of $\Spec R$.
Then it follows from Proposition \ref{Enochs theorem} that $F$ is isomorphic to an $R$-module of the form $\prod_{\fp\in W} T_\fp$.
More precisely, using Lemma \ref{characterization complete flat}, (\ref{colocalization of cotorsion flat}) and $(\ref{completion of cotorsion flat})$, one can show the following corollary, which is essentially proved in \cite[Lemma 8.5.25]{EJ}.

\begin{corollary}\label{Enochs cor}
Let $F$ be a cotorsion flat $R$-module, and set $W=\cosupp F$. Then we have an isomorphism
$$F\cong \prod_{\fp \in W} T_\fp,$$
where $T_\fp$ is of the form $(\bigoplus_{B_\fp} R_\fp)^\wedge_\fp$ with $B_\fp=\dim_{\kappa(\fp)}\Hom_R(R_\fp, F)\otimes_R\kappa(\fp)$. 
\end{corollary}

\section{Complexes of cotorsion flat modules and cosupport}

In this section, we study the cosupport of a complex $X$ consisting of cotorsion flat $R$-modules.
As a consequence, we obtain an explicit way to calculate $\gamma_{V^c}X$ and $\lambda^VX$ for a specialization-closed subset $V$ of $\Spec R$.

\begin{notation}\label{Enochs rep}
Let $W$ be a subset of $\Spec R$.
Let $X$ be a complex of cotorsion flat $R$-modules such that $\cosupp X^i\subseteq W$ for all $i\in \mathbb{Z}$.
Under Corollary \ref{Enochs cor}, we use a presentation of the following form;
$$X=(\cdots\to  \prod_{\fp\in W}T^i_\fp\to \prod_{\fp\in W}T^{i+1}_\fp\to \cdots ),$$ 
where $X^i= \prod_{\fp\in \Spec R}T^i_\fp$ and $T^i_\fp$ is the $\fp$-adic completion of a free $R_\fp$-module.
\end{notation}

\begin{remark}\label{remark for subcomplex}
Let $X=(\cdots\to  \prod_{\fp\in \Spec R}T^i_\fp\to \prod_{\fp\in \Spec R}T^{i+1}_\fp\to \cdots )$ be a complex of cotorsion flat $R$-modules. 
Let $V$ be a specialization-closed subset of $\Spec R$.
By Lemma \ref{localizing and colocalizing}, we have $\Hom_R(\prod_{\fp\in V^c}T^i_\fp,\prod_{\fp\in V}T^{i+1}_\fp)=0$ for all $i\in \mathbb{Z}$. 
Therefore $Y=(\cdots\to  \prod_{\fp\in V^c}T^i_\fp\to \prod_{\fp\in V^c}T^{i+1}_\fp\to \cdots )$
 is a subcomplex of $X$, where the differentials in $Y$ are the restrictions of ones in $X$.
\end{remark}

We say that a complex $X$ of $R$-modules is left (resp. right) bounded if $X^i=0$ for $i\ll 0$ (resp. $i\gg 0$). When $X$ is left and right bounded, $X$ is called bounded.

\begin{proposition}\label{cosupport key prop}
Let $W$ be a subset of $\Spec R$ and $X$ be a complex of cotorsion flat $R$-modules such that $\cosupp X^i\subseteq W$ for all $i\in \mathbb{Z}$.
Suppose that one of the following conditions holds;
\vspace{1mm}
\begin{enumerate}[label={\rm(\arabic*)}]
\item $X$ is left bounded;
\item $W$ is equal to $V(\fa)$ for an ideal $\fa$ of $R$; 
\item $W$ is generalization-closed;
\item $\dim W$ is finite.
\end{enumerate}
\vspace{1mm}
Then it holds that $\cosupp X\subseteq W$, {\it i.e.}, $X\in \cC^W$.
\end{proposition}

To prove this proposition, we use the next elementary lemma.
In the lemma, for a complex $X$ and  $n\in \mathbb{Z}$, we define the truncations $\tau_{\leq n}X$ and $\tau_{> n}X$ as follows (cf. \cite[Chapter I, \S 7]{Ha}):
\begin{align*}
\tau_{\leq n}X&=(\cdots\to X^{n-1}\to X^{n}\to 0\to\cdots )\\
\tau_{> n}X&=(\cdots\to 0\to X^{n+1}\to X^{n+2}\to\cdots )
\end{align*}

\begin{lemma}\label{basic truncation lemma}
Let $W$ be a subset of $\Spec R$.
We assume that $\tau_{\leq n} X\in \cC^W$ (resp. $\tau_{> n}X\in \cL_W$) for all $n\geq 0$ (resp. $n< 0$).
Then we have $X\in  \cC^W$ (resp. $X\in \cL_W$).
\end{lemma}

Recall that $\cC^W$ (resp. $\cL_W$) is closed under arbitrary direct products (resp. sums). Then one can show this lemma by using homotopy limit (resp. colimit), see \cite[Remark 2.2, Remark 2.3]{BN}.\qed\bigskip

\noindent {\it Proof of Proposition \ref{cosupport key prop}.}
Case (1): We have $\tau_{\leq n}X\in \cC^W$ for all $n\geq 0$, since $\tau_{\leq n}X$ are bounded.
Thus Lemma \ref{basic truncation lemma} implies that $X\in \cC^W$. 

Case (2): By (\ref{derived local homology}), Proposition \ref{not necessarily K-flat} and (\ref{completion of cotorsion flat}),
it holds that $\lambda^{V(\fa)}X\cong \LLambda^{V(\fa)}X\cong \Lambda^{V(\fa)}X\cong X$.
Hence $X$ belongs to $\cC^{V(\fa)}$.

Case (3): It follows from the case (1) that $\tau_{> n}X\in \cC^W$ for all $n<0$.
Moreover, we have $\cC^W=\cL_W$ by Lemma \ref{localizing and colocalizing}. 
Thus Lemma \ref{basic truncation lemma} implies that $X\in \cL_W=\cC^W$.

Case (4): Under Notation \ref{Enochs rep}, we write $X^i=\prod_{\fp\in W}T^i_\fp$ for $i\in \mathbb{Z}$.
Set $n=\dim W$, and use induction on $n$.
First, suppose that $n=0$.
It is seen from Remark \ref{remark for subcomplex} that 
$X$ is the direct product of complexes of the form $ Y^{\{\fp\}}=(\cdots\to T_\fp^i\to T_\fp^{i+1}\to\cdots)$ for $\fp\in W$.
Furthermore, by the cases (2) and (3), we have $\cosupp Y^{\{\fp\}}\subseteq V(\fp)\cap U(\fp)=\{\fp\}$.
Thus it holds that $X\cong \prod_{\fp\in W}Y^{\{\fp\}}\in \cC^W$.

Next, suppose that $n>0$. Set $W_0=\max W$ and $W_1=W\backslash W_0$.
We write $Y=(\cdots\to  \prod_{\fp\in W_1}T^i_\fp\to \prod_{\fp\in W_1}T^{i+1}_\fp\to \cdots )$, which is a subcomplex of $X$ by Remark \ref{remark for subcomplex}. 
Hence there is a short exact sequence of complexes;
$$\begin{CD}
0@>>>Y@>>>X@>>>X/Y@>>>0,
\end{CD}$$
where
$X/Y=(\cdots\to  \prod_{\fp\in W_0}T^i_\fp\to \prod_{\fp\in W_0}T^{i+1}_\fp\to \cdots ).$
Note that $\dim W_0=0$ and $\dim W_1=n-1$.
Then we have $\cosupp X/Y\subseteq W_0$, by the argument above.
Moreover the inductive hypothesis implies that $\cosupp Y\subseteq W_1$.
Hence it holds that $\cosupp X\subseteq W_0\cup W_1=W$.
\qed\medskip

Under some assumption, it is possible to extend the condition (4) in Proposition \ref{cosupport key prop}
to the case where $\dim W$ is infinite, see Remark \ref{infinite case}. See also \cite[Theorem 2.5]{PT1}.

\begin{corollary}\label{compute lambda}
Let $X$ be a complex of cotorsion flat $R$-modules and $W$ be a specialization-closed subset of $\Spec R$.
Under Notation \ref{Enochs rep}, we write 
$$X=(\cdots \to \prod_{\fp\in \Spec R} T^i_\fp \to \prod_{\fp\in \Spec R}  T^{i+1}_\fp\to \cdots ).$$ 
Suppose that one of the conditions in Proposition \ref{cosupport key prop} holds.
Then it holds that
\begin{align*}
\gamma_{W^c}X\cong& (\cdots \to \prod_{\fp\in W^c} T^i_\fp\to \prod_{\fp\in W^c} T^{i+1}_\fp\to \cdots ),\\
\lambda^{W}X\cong& (\cdots \to \prod_{\fp\in W} T^i_\fp \to \prod_{\fp\in W}  T^{i+1}_\fp\to \cdots ).
\end{align*}
\end{corollary}

\begin{proof}
Since $Y=(\cdots \to \prod_{\fp\in W^c} T^i_\fp\to \prod_{\fp\in W^c} T^{i+1}_\fp\to \cdots )$ is a subcomplex of $X$ by Remark \ref{remark for subcomplex}, there is a triangle in $\cD$;
$$\begin{CD}
Y@>>>X@>>>X/Y@>>>Y[1],
\end{CD}$$
where $X/Y=(\cdots \to \prod_{\fp\in W} T^i_\fp \to \prod_{\fp\in W}  T^{i+1}_\fp\to \cdots )$.
By Proposition \ref{cosupport key prop}, we have $X/Y\in \cC^W$.
Moreover, since $W^c$ is generalization-closed, it holds that $Y\in \cC^{W^c}={}^\perp \cC^{W}$ by Proposition \ref{cosupport key prop} and Lemma \ref{localizing and colocalizing}.
Therefore we conclude that $\gamma_{W^c}X\cong Y$ and $\lambda^{W}X\cong X/Y$ by Lemma \ref{basic lemma}.
\end{proof}

Let $X$ be a complex of cotorsion flat $R$-modules and $S$ be a multiplicatively closed subset of $R$.
We assume that $X$ is left bounded, or $\dim R$ is finite. 
It then follows from the corollary and (\ref{colocalization of cotorsion flat}) that 
$$\gamma_{U_S}X\cong (\cdots \to \prod_{\fp\in U_S} T^i_\fp\to \prod_{\fp\in U_S} T^{i+1}_\fp\to \cdots )\cong\Hom_R(S^{-1}R,X).$$
We now recall that $\gamma_{U_S} \cong \RHom_R(S^{-1}R,-)$, see \cite[Proposition 3.1]{NY}.
Hence it holds that $\RHom_R(S^{-1}R,X)\cong \Hom_R(S^{-1}R, X)$. This fact also follows from Lemma \ref{flat and pure-injective}.

\section{Localization functors via \v{C}ech complexes}

In this section, we introduce a new notion of \v{C}ech complexes to calculate $\lambda^WX$, where $W$ is a general subset $W$ of $\Spec R$ and $X$ is a complex of flat $R$-modules.\medskip

We first make the following notations.

\begin{notation}
Let $W$ be a subset of $\Spec R$ with $\dim W=0$.
We define a functor $\bar{\lambda}^{W}:\Mod R\to \Mod R$ by
$$\bar{\lambda}^{W}=\prod_{\fp\in W}\Lambda^{V(\fp)} (-\otimes_R R_\fp).$$
For a prime ideal $\fp$ in $W$, 
we write $\bar{\eta}^{\{\fp\}}:\id_{\Mod R}\to  \bar{\lambda}^{\{\fp\}}=\Lambda^{V(\fp)} (-\otimes_R R_\fp)$ for the composition of the natural morphisms $\id_{\Mod R}\to (-)\otimes_R R_\fp$ and $(-)\otimes_R R_\fp \to \Lambda^{V(\fp)} (-\otimes_R R_\fp)$.
Moreover, $\bar{\eta}^{W}:\id_{\Mod R}\to \bar{\lambda}^{W}=\prod_{\fp\in W}\bar{\lambda}^{\{\fp\}}$ denotes the product of the morphisms $\bar{\eta}^{\{\fp\}}$ for $\fp\in W$.
\end{notation}

\begin{notation}\label{Cech notation}
Let $\{W_i\}_{0\leq i\leq n}$ be a family of subsets of $\Spec R$, and suppose that $\dim W_i=0$ for $0\leq i\leq n$.
For a sequence $(i_m,\ldots i_1, i_0)$ of integers  with $0\leq i_0<i_1<\cdots<i_m\leq n$, 
we write
$$\bar{\lambda}^{(i_m,\ldots i_1, i_0)}=\bar{\lambda}^{W_{i_m}}\cdots\bar{\lambda}^{W_{i_1}}\bar{\lambda}^{W_{i_0}}.$$
If the sequence is empty, then we use the general convention that $\lambda^{(\ )}=\id_{\Mod R}$.
For an integer $s$ with $0\leq s\leq m$, 
$\bar{\eta}^{W_{i_s}}:\id_{\Mod R}\to \bar{\lambda}^{(i_s)}$ induces a morphism 
$$\bar{\lambda}^{(i_m,\ldots, i_{s+1})}\bar{\eta}^{W_{i_s}}\bar{\lambda}^{(i_{s-1},\ldots, i_{0})}:\bar{\lambda}^{(i_m,\ldots,\widehat{i_s},\ldots, i_{0})}\longrightarrow \bar{\lambda}^{(i_m,\ldots, i_{0})},$$
where we mean by $\widehat{i_s}$ that $i_s$ is omitted.
We set 
$$\partial^{m-1}: \prod_{0\leq i_0<\cdots<i_{m-1}\leq n}\bar{\lambda}^{(i_{m-1},\ldots, i_0)}\longrightarrow \prod_{0\leq i_0<\cdots<i_{m}\leq n}\bar{\lambda}^{(i_{m},\ldots, i_0)}$$
to be the product of the morphisms $\bar{\lambda}^{(i_m,\ldots,\widehat{i_s},\ldots, i_{0})}\to \bar{\lambda}^{(i_m,\ldots, i_{0})}$ multiplied by $(-1)^s$.
\end{notation}

\begin{remark}\label{remark cech complex}
Let $W_0, W_1 \subseteq \Spec R$ be subsets such that $\dim W_0=\dim W_1=0$.
As with Remark \ref{remark for main result 1} (ii), the following diagram is commutative:
$$\begin{CD}
\id_{\Mod R}@>\bar{\eta}^{W_0}>>\bar{\lambda}^{W_0}\\
@VV\bar{\eta}^{W_1}V@VV\bar{\eta}^{W_1}\bar{\lambda}^{W_0}V\\
\bar{\lambda}^{W_1}@>\bar{\lambda}^{W_1}\bar{\eta}^{W_0}>>\bar{\lambda}^{W_1}\bar{\lambda}^{W_0}
\end{CD}$$
\end{remark}\medskip

\begin{definition}\label{Cech def}
Let $\mathbb{W}=\{W_i\}_{0\leq i\leq n}$ be a family of subsets of $\Spec R$, and suppose that $\dim W_i=0$ for $0\leq i\leq n$.
By Remark \ref{remark cech complex}, it is possible to construct a \v{C}ech complex of functors of the following form;
$$\prod_{0\leq i_0\leq n}\bar{\lambda}^{(i_0)}\xrightarrow{\partial^0}\prod_{0\leq i_0<i_1\leq n}\bar{\lambda}^{(i_1,i_0)}\to\cdots\to  \prod_{0\leq i_0<\cdots<i_{n-1}\leq n}\bar{\lambda}^{(i_{n-1},\ldots, i_0)}\xrightarrow{\partial^{n-1}}  \bar{\lambda}^{(n,\ldots, 0)},$$
which we denote by $L^{\mathbb{W}}$ and call it {\it the \v{C}ech complex with respect to $\mathbb{W}$}.

For an $R$-module $M$, $L^{\mathbb{W}}M$ denotes the complex of $R$-modules obtained by $L^{\mathbb{W}}$ in a natural way, where it is concentrated in degrees from $0$ to $n$.
We call $L^{\mathbb{W}}M$ {\it the \v{C}ech complex of $M$ with respect to $\mathbb{W}$}.
Note that there is a chain map $\ell^{\mathbb{W}}M:M\to  L^{\mathbb{W}}M$ induced by the map $M\to \prod_{0\leq i_0\leq n}\bar{\lambda}^{(i_0)}M$ in degree $0$, which is the product of $\bar{\eta}
^{W_{i_0}}M:M\to \bar{\lambda}^{(i_0)}M$ for $0\leq i_0\leq n$.

More generally, we regard every term of $L^{\mathbb{W}}$ as a functor $C(\Mod R)\to C(\Mod R)$, where $C(\Mod R)$ denotes the category of complexes of $R$-modules.
Then $L^{\mathbb{W}}$ naturally sends a complex $X$ to a double complex, which we denote by  $L^{\mathbb{W}}X$. 
Furthermore, we write $\tot L^{\mathbb{W}}X$ for the total complex of $L^{\mathbb{W}}X$. 
The family of chain maps $\ell^{\mathbb{W}}X^j:X^j\to L^{\mathbb{W}}X^j$ for $j\in\mathbb{Z}$ induces a morphism $X\to L^{\mathbb{W}}X$ as double complexes, from which we obtain a chain map $\ell^{\mathbb{W}}X:X\to \tot L^{\mathbb{W}}X$.
\end{definition}

\begin{remark}\label{remark natural chain map}
(i) We regard $\tot L^{\mathbb{W}}$ as a functor $C(\Mod R)\to C(\Mod R)$.
Then $\ell^{\mathbb{W}}$ is a morphism $\id_{C(\Mod R)}\to \tot L^{\mathbb{W}}$ of functors.
Moreover, if $M$ is an $R$-module, then $\tot L^{\mathbb{W}}M=L^{\mathbb{W}}M$.

(ii) Let $a,b \in  \Z\cup \{\pm\infty\}$ with $a\leq b$ and  $X$ be a complex of $R$-modules such that 
$X^i=0$ for $i\notin [a,b]$.
Then it holds that $(\tot L^{\mathbb{W}}X)^i=0$ for $i\notin [a,b+n]$, where $n$ is the number given to  $\mathbb{W}=\{W_i\}_{0\leq i\leq n}$.

(iii) Let $X$ be a complex of flat $R$-modules.
Then we see that $\tot L^{\mathbb{W}}X$ consists of cotorsion flat $R$-modules with cosupports in $\bigcup_{0\leq i\leq n}W_i$.
\end{remark}

\begin{definition}\label{star-partition}
Let $W$ be a non-empty subset of $\Spec R$ and $\{W_i\}_{0\leq i\leq n}$ be a family of subsets of $W$.
We say that $\{W_i\}_{0\leq i\leq n}$ is {\it a system of slices of $W$} if the following conditions hold:
\begin{enumerate}
\item $W=\bigcup_{0\leq i\leq n}W_i$;
\item $W_i\cap W_j=\emptyset$ if $i\neq j$;
\item $\dim W_i= 0$ for $0\leq i\leq n$;
\item $W_i$ is specialization-closed in $\bigcup_{i\leq j\leq n}W_j$ for each $0\leq i \leq n$.
\end{enumerate}
\end{definition}\medskip

Compare this definition with the filtrations discussed in \cite[Chapter IV; \S 3]{Ha}.

If $\dim W$ is finite, then there exists at least one system of slices of $W$.
Conversely, if there is a system of slices of $W$, then $\dim W$ is finite.

\begin{proposition}\label{main result 2}
Let $W$ be a subset of $\Spec R$ and 
$\mathbb{W}=\{W_i\}_{0\leq i\leq n}$ be a system of slices of $W$.
Then, for any flat $R$-module $F$, there is an isomorphism in $\cD$;
$$\lambda^WF\cong L^{\mathbb{W}}F.$$
Under this isomorphism, $\ell^{\mathbb{W}}F:F\to L^{\mathbb{W}}F$ coincides with $\eta^{W}F:F\to \lambda^{W}F$ in $\cD$.
\end{proposition}

\begin{proof}
We use induction on $n$, which is the number given to $\mathbb{W}=\{W_i\}_{0\leq i\leq n}$.
Suppose that $n=0$. It then holds that $L^{\mathbb{W}}F=\bar{\lambda}^{W_0}F=\bar{\lambda}^{W}F$ and $\ell^{\mathbb{W}}F=\bar{\eta}^{W_0}F=\bar{\eta}^WF$.
Hence this proposition follows from Theorem \ref{lambda dim W=0}.

Next, suppose that $n>0$, and write $U=\bigcup_{1\leq i\leq n}W_i$. 
Setting $U_{i-1}=W_{i}$, we obtain a system of slices $\mathbb{U}=\{U_i\}_{0\leq i\leq n-1}$ of $U$.
Consider the following two squares, where the first (resp. second) one is in $C(\Mod R)$ (resp. $\cD$): 

$$\begin{CD}
F@>\bar{\eta}^{W_0}F>> \bar{\lambda}^{W_0}F\\
@VV\ell^{\mathbb{U}}FV@VV\ell^{\mathbb{U}}\bar{\lambda}^{W_0}FV\\
L^{\mathbb{U}}F@>L^{\mathbb{U}}\bar{\eta}^{W_0}F>> L^{\mathbb{U}}\bar{\lambda}^{W_0}F\\
\end{CD}
\hspace{2cm}
\begin{CD}
F@>\eta^{W_0}F>>\lambda^{W_0}F \\
@VV\eta^{U}FV@VV\eta^{U}\lambda^{W_0}FV\\
\lambda^{U}F@>\lambda^{U}\eta^{W_0}F>> \lambda^{U}\lambda^{W_0}F
\end{CD}\vspace{2mm}$$
By Remark \ref{remark natural chain map} (i) and Remark \ref{remark for main result 1} (ii), both of them are commutative. 
Moreover, $\lambda^{U}\eta^{W_0}F$ is the unique morphism which makes the right square commutative, because $\lambda^{U}$ is a left adjoint to the inclusion functor $\cC^U\hookrightarrow \cD$. 
Then, regarding the left one as being in $\cD$, we see from the inductive hypothesis that the left one coincides with the right one in $\cD$.

Let $\bar{g}:L^{\mathbb{U}}F\oplus\bar{\lambda}^{W_0}F\to L^{\mathbb{U}}\bar{\lambda}^{W_0}F$ and $\bar{h}:F\to L^{\mathbb{U}}F\oplus\bar{\lambda}^{W_0}F$ be chain maps represented by the following matrices;  
$$\bar{g}=\left(\begin{array}{cc}L^{\mathbb{U}}\bar{\eta}^{W_0}F\  \ \ 
(-1)\cdot\ell^{\mathbb{U}}\bar{\lambda}^{W_0}X
\end{array}\right),\  \bar{h}=\left(\begin{array}{c}
\ell^{\mathbb{U}}F
\\[7pt]
\bar{\eta}^{W_0}F
\end{array}\right).$$
Notice that the mapping cone of $\bar{g}[-1]$ is nothing but $L^\mathbb{W}F$.
Then we can obtain the following morphism of triangles, where it is regarded as being in $\cD$:
\begin{align}
\begin{CD}
F[-1]@>>>F[-1]@>>>0@>>>F\\
@VV\ell^{\mathbb{W}}F[-1]V@VV\bar{h}[-1]V@VVV@VV\ell^{\mathbb{W}}FV\\
L^\mathbb{W}F[-1]@>>> (L^{\mathbb{U}}F\oplus\bar{\lambda}^{W_0}F)[-1] @>\bar{g}[-1]>>  L^{\mathbb{U}}\bar{\lambda}^{W_0}F[-1] @>>> L^\mathbb{W}F\label{diagram in main theorem1}
\end{CD}
\end{align}
Therefore, by Theorem \ref{main result 1} and Remark \ref{remark for the natural morphism}, there is an isomorphism $\lambda^{W}F\cong L^\mathbb{W}F$ such that $\ell^{\mathbb{W}}F$ coincides with $\eta^{W}F$ under this isomorphism.
 \end{proof}\medskip

The following corollary is one of the main results of this paper.

\begin{corollary}\label{cor main result 2}
Let $W$ and $\mathbb{W}=\{W_i\}_{0\leq i\leq n}$ be as above. 
Let $X$ be a complex of flat $R$-modules.
Then there is an isomorphism in $\cD$;
$$\lambda^WX\cong \tot L^{\mathbb{W}}X.$$
Under this isomorphism, $\ell^{\mathbb{W}}X:X\to\tot L^{\mathbb{W}}X$ coincides with  $\eta^WX:X\to \lambda^WX$ in $\cD$.
\end{corollary}

\begin{proof}
We embed $\ell^{\mathbb{W}}X:X\to \tot L^{\mathbb{W}}X$ into a triangle 
$$\begin{CD}
C@>>> X@>\ell^{\mathbb{W}}X>> \tot L^{\mathbb{W}}X@>>> C[1].
\end{CD} $$
Proposition \ref{cosupport key prop} and Remark \ref{remark natural chain map} (iii) imply that $\tot L^{\mathbb{W}}X\in \cC^W$.
Thus it suffices to show that $\lambda^{W_i}C=0$ for each $i$, by  Lemma \ref{basic lemma} and Remark \ref{characterization perp} (i).
For this purpose, we prove that $\lambda^{W_i}\ell^{\mathbb{W}}X$ is an isomorphism in $\cD$.
This is equivalent to showing that $\bar{\lambda}^{W_i}\ell^{\mathbb{W}}X$ is a quasi-isomorphism, since $X$ and $\tot L^{\mathbb{W}}X$ consist of flat $R$-modules.

Consider the natural morphism $X\to L^{\mathbb{W}}X$ as double complexes, which is induced by the chain maps $\ell^{\mathbb{W}}X^j:X^j\to L^{\mathbb{W}}X^j$ for $j\in \mathbb{Z}$.
To prove that $\bar{\lambda}^{W_i}\ell^{\mathbb{W}}X$ is a quasi-isomorphism, it is enough to show that $\bar{\lambda}^{W_i}\ell^{\mathbb{W}}X^j$ is a quasi-isomorphism for each $j\in \mathbb{Z}$, see \cite[Theorem 12.5.4]{KS}. 
Furthermore, by Proposition \ref{main result 2}, each $\ell^{\mathbb{W}}X^j$ coincides with $\eta^WX^j:X^j\to \lambda^WX^j$ in $\cD$.
Since $W_i\subseteq W$, it follows from Remark \ref{characterization perp} (ii) that $\lambda^{W_i}\eta^{W}X^j$ is an isomorphism in $\cD$. This means that $\bar{\lambda}^{W_i}\ell^{\mathbb{W}}X^j$ is a quasi-isomorphism.
\end{proof}\medskip

Let $W$ be a subset of $\Spec R$, and suppose that $n=\dim W$ is finite. Then Corollary \ref{cor main result 2} implies $\lambda^{W}R\in \cD^{[0,n]}$. 
We give an example such that $H^{n}(\lambda^{W}R)\neq 0$.

\begin{example}\label{positive cohomology}
Let $(R, \fm)$ be a local ring of dimension $d\geq 1$. Then we have $\dim V(\fm)^c= d-1$.
By Lemma \ref{basic lemma}, there is a triangle
$$\begin{CD}
\gamma_{V(\fm)}R@>>>R@>>>\lambda^{V(\fm)^c}R@>>>\gamma_{V(\fm)}R[1].
\end{CD}$$
Since $\RGamma_{V(\fm)}\cong \gamma_{V(\fm)}$ by (\ref{derived local cohomology}),
Grothendieck's non-vanishing theorem implies that $H^d(\gamma_{V(\fm)}R)$ is non-zero.
Then we see from the triangle that $H^{d-1}(\lambda^{V(\fm)^c}R)\neq 0$.
\end{example}\medskip

We denote by $\cD^{-}$ the full subcategory 
of $\cD$ consisting of complexes $X$ such that $H^i(X)=0$ for $i\gg 0$.
Let $W$ be a subset of $\Spec R$ and $X\in \cD^-$.
If $\dim W$ is finite, then we have $\lambda^WR\in \cD^{-}$ by Corollary \ref{cor main result 2}.
However,  as shown in the following example, it can happen that $\lambda^WR\notin \cD^{-}$ when $\dim W$ is infinite.

\begin{example}\label{infinite cohomology}
Assume that $\dim R=+\infty$, and set $W=\max(\Spec R)$.
Then it holds that $\dim W=0$ and $\dim W^c=+\infty$.
Since each $\fm\in W$ is maximal, there are isomorphisms 
$$\gamma_{W}\cong \RGamma_W\cong \bigoplus_{\fm\in W}\RGamma_{V(\fm)}.$$ 
Thus we see from Example \ref{positive cohomology} that $\gamma_WR\notin \cD^{-}$. 
Then, considering the triangle
$$\begin{CD}
\gamma_{W}R@>>>R@>>>\lambda^{W^c}R@>>>\gamma_{W}R[1],
\end{CD}$$
we have $\lambda^{W^c}R\notin\cD^{-}$. 
\end{example}\medskip

Let $W$ be a subset of $\Spec R$ and $X\in \cC^W$.
Then $\eta^WX:X\to \lambda^WX$ is an isomorphism in $\cD$. 
Thus Remark \ref{remark natural chain map} (iii) and Corollary \ref{cor main result 2} yield the following result.

\begin{corollary}\label{[a,b] general}
Let $W$ be a subset of $\Spec R$, and $\mathbb{W}=\{W_i\}_{0\leq i\leq n}$ be a system of slices of $W$.
Let $X$ be a complex of flat $R$-modules with $\cosupp X\subseteq W$.
Then the chain map $\ell^{\mathbb{W}}X:X\to \tot L^{\mathbb{W}}X$ is a quasi-isomorphism, where $\tot L^{\mathbb{W}}X$ consists of cotorsion flat $R$-modules with cosupports in $W$.
\end{corollary}

\begin{remark}
If $d=\dim R$ is finite, then any complex $Y$ is quasi-isomorphic to a $K$-flat complex consisting of cotorsion flat $R$-modules. 
To see this, set $W_{i}=\Set{\fp\in \Spec R | \dim R/\fp= i}$ for $0\leq i\leq d$. Then $\mathbb{W}=\{W_i\}_{0\leq i\leq d}$ is a system of slices of $\Spec R$.
We take a $K$-flat resolution $X$ of $Y$ such that $X$ consists of flat $R$-modules. 
Corollary \ref{[a,b] general} implies that 
$\ell^{\mathbb{W}}X:X\to \tot L^\mathbb{W}X$ is a quasi-isomorphism, and $\tot L^\mathbb{W}X$ consists of cotorsion flat $R$-modules.
At the same time, the chain maps $\ell^{\mathbb{W}}X^i:X^i\to L^\mathbb{W}X^i$ are quasi-isomorphisms for all $i\in \mathbb{Z}$. 
Then it is not hard to see that the mapping cone of $\ell^{\mathbb{W}}X$ is $K$-flat.
Thus $\tot L^{\mathbb{W}}X$ is $K$-flat.
\end{remark}

By Proposition \ref{cosupport key prop} and Corollary \ref{[a,b] general}, we have the next result. 

\begin{corollary}\label{if and only if}
Let $W$ be a subset of $\Spec R$ such that $\dim W$ is finite. Then a complex $X\in \cD$ belongs to $\cC^W$ if and only if $X$ is isomorphic to a complex $Z$ of cotorsion flat $R$-modules such that $\cosupp Z^i\subseteq W$ for all $i\in \mathbb{Z}$.
\end{corollary}

\begin{remark}\label{infinite case}
If $\dim W$ is infinite, it is possible to construct a similar family to systems of slices.
We first put $W_{0}=\max W$. 
Let $i>0$ be an ordinal, and suppose that subsets $W_{j}$ of $W$ are defined for all $j<i$.
Then we put $W_{i}=\max (W\backslash  \bigcup_{j<i}W_{j})$.
In this way, we obtain the smallest ordinal $o(W)$ satisfying the following conditions: (1) $W=\bigcup_{0 \leq i<o(W)} W_i$; (2) $W_i\cap W_j=\emptyset$ if $i\neq j$; (3) $\dim W_i\leq 0$ for $0\leq i<o(W)$; (4) $W_i$ is specialization-closed in $\bigcup_{i\leq j<o(W)}W_j$ for each $0\leq i<o(W)$. 

One should remark that the ordinal $o(W)$ can be uncountable in general, see \cite[p. 48, Theorem 9.8]{GR}.
However, if $R$ is an infinite dimensional commutative Noetherian ring given by Nagata \cite[Appendix A1; Example 1]{Nag}, then $o(W)$ is at most countable.
Moreover, using transfinite induction, it is possible to extend the condition (4) in Proposition \ref{cosupport key prop} and Corollary  \ref{compute lambda} to the case where $o(W)$ is countable. 
One can also extend 
Corollary \ref{if and only if} to the case where $o(W)$ is countable.
\end{remark}\medskip

Using Theorem \ref{gamma pushout} and results in \cite[\S 3]{NY}, it is possible to give a similar result to Corollary \ref{cor main result 2}, for colocalization functors $\gamma_W$ and complexes of injective $R$-modules.\medskip

\section{\v{C}ech complexes and complexes of finitely generated modules}
Let $W$ be a subset of $\Spec R$ and 
$\mathbb{W}=\{W_i\}_{0\leq i\leq n}$ be a system of slices of $W$.
In this section, we prove that $\lambda^WY$ is isomorphic to $\tot L^{\mathbb{W}}Y$ if $Y$ is a complex of finitely generated $R$-modules.\medskip

We denote by $\cD_{\rm fg}$ the full subcategory of $\cD$ consisting of all complexes with finitely generated cohomology modules, and  set $\cD^{-}_{\rm fg}=\cD^{-}\cap \cD_{\rm fg}$. 
We first prove the following proposition.

\begin{proposition}\label{commute with tensor}
Let $W$ be a subset of $\Spec R$ such that $\dim W$ is finite.
Let $X, Y\in \cD$. We suppose that one of the following conditions holds;
\begin{enumerate}[label={\rm(\arabic*)}]
\item $X\in \cD^{-}$ and $Y\in \cD^{-}_{\rm fg}$; 
\item $X$ is a bounded complex of flat $R$-modules and $Y\in \cD_{\rm fg}$.
\end{enumerate}
Then there are natural isomorphisms 
$$(\gamma_{W^c}X) \Lotimes Y\cong \gamma_{W^c}(X \Lotimes Y), \ \  (\lambda^{W}X)\Lotimes Y\cong \lambda^{W}(X\Lotimes Y).$$ 
\end{proposition}\medskip

For $X\in \cD$ and $n\in \mathbb{Z}$, we define the cohomological truncations $\sigma_{\leq n}X$ and $\sigma_{> n}X$ as follows (cf. \cite[Chapter I; \S 7]{Ha}):
\begin{align*}\sigma_{\leq n}X&=(\cdots \to X^{n-2}\to X^{n-1}\to \Ker d^{n}_X\to 0\to\cdots )\\
\sigma_{> n}X&=(\cdots\to 0\to \Image d_X^{n}\to X^{n+1}\to X^{n+2}\to\cdots )
\end{align*}

\noindent{\it Proof of Proposition \ref{commute with tensor}.}
Apply $(-)\Lotimes Y$ to the triangle
$\gamma_{W^c}X\to X\to \lambda^WX\to  \gamma_{W^c}X[1]$, and we obtain the following triangle;
$$
\begin{CD}
(\gamma_{W^c}X) \Lotimes Y@>>> X\Lotimes Y@>>> (\lambda^{W}X)\Lotimes Y@>>> (\gamma_{W^c}X) \Lotimes Y[1].
\end{CD}$$
Since $\supp \gamma_{W^c}X \subseteq W^c$, we have $\supp (\gamma_{W^c}X) \Lotimes Y\subseteq W^c$, that is, $(\gamma_{W^c}X) \Lotimes Y\in \cL_{W^c}$. 
Hence it remains to show that $(\lambda^{W}X)\Lotimes Y\in \cC^W$, see Lemma \ref{basic lemma}.

Case (1): 
We remark that $X$ is isomorphic to a right bounded complex of flat $R$-modules.
Then it is seen from Corollary \ref{cor main result 2} that $\lambda^WX$ is isomorphic to a right bounded complex $Z$ of cotorsion flat $R$-modules such that $\cosupp Z^i\subseteq W$ for all $i\in \mathbb{Z}$. Furthermore, $Y$ is isomorphic to a right bounded complex $P$ of finite free $R$-modules. 
Hence it follows that 
$X\Lotimes Y\cong Z\otimes_RP$, where the second one consists of cotorsion flat $R$-modules with cosupports in $W$.
Then we have $X\Lotimes Y\cong Z\otimes_RP\in \cC^W$ by Proposition \ref{cosupport key prop}.

Case $(2)$:
By Corollary \ref{cor main result 2}, $\lambda^{W}X$ is isomorphic to a bounded complex consisting of cotorsion flat $R$-modules with cosupports in $W$.
Thus it is enough to prove that $Z\otimes_R Y\in \cC^W$ for a cotorsion flat $R$-module $Z$ with $\cosupp Z\subseteq W$.

We consider the triangle 
$\sigma_{\leq n}Y\to Y\to \sigma_{>n}Y\to \sigma_{\leq n}Y[1]$ for an integer $n$.
Applying $Z\otimes_R(-)$ to this triangle, we obtain the following one;
$$\begin{CD}
Z\otimes_R\sigma_{\leq n}Y@>>>Z\otimes_RY@>>>Z\otimes_R\sigma_{>n}Y@>>>Z\otimes_R\sigma_{\leq n}Y[1].
\end{CD}$$
Let $\fp\in W^c$.
The case (1) implies that  
$Z\otimes_R \sigma_{\leq n}Y\in \cC^{W}$ for any $n\in \mathbb{Z}$, since $\lambda^WZ\cong Z$. 
Thus, applying $\RHom_R(\kappa(\fp),-)$ to the triangle above, 
we have
$$\RHom_R(\kappa(\fp), Z\otimes_RY)\cong \RHom_R(\kappa(\fp), Z\otimes_R\sigma_{>n}Y).$$
Furthermore, taking a projective resolution $P$ of $\kappa(\fp)$,
we have
$$\RHom_R(\kappa(\fp), Z\otimes_R\sigma_{>n}Y)\cong \Hom_R(P, Z\otimes_R\sigma_{>n}Y).$$
Let $j$ be any integer. 
To see that $\RHom_R(\kappa(\fp), Z\otimes_RY)=0$, it suffices to show that there exists an integer $n$ 
such that $H^0(\Hom_R(P[j], Z\otimes_R\sigma_{>n}Y))=0$.
Note that $P^i=0$ for $i>0$.
Moreover, each element of $H^0(\Hom_R(P[j], Z\otimes_R\sigma_{>n}Y))\cong \Hom_{\cD}(P[j], Z\otimes_R\sigma_{>n}Y)$ is represented by a chain map $P[j]\to Z\otimes_R\sigma_{>n}Y$.
Therefore it holds that $H^0(\Hom_R(P[j], Z\otimes_R\sigma_{>n}Y))=0$ if $n>-j$.
\qed

\begin{remark}
(i) In the proposition, we can remove the finiteness condition on $\dim W$ if $W=V(\fa)$ for an ideal $\fa$.
In such case, we need only use $\fa$-adic completions of free $R$-modules instead of cotorsion flat $R$-modules.

(ii) If $W$ is a generalization-closed subset of $\Spec R$, then the isomorphisms in the proposition hold for any $X, Y\in \cD$ because $\gamma_{W^c}$ is isomorphic to $\RGamma_{W^c}$.
\end{remark}\vspace{2mm}

Let $W$ be a subset of $\Spec R$ and $\mathbb{W}=\{W_i\}_{0\leq i\leq n}$ be a system of slices of $W$.
Let $Y\in \cD_{\rm fg}$. 
By Proposition \ref{commute with tensor} and Proposition \ref{main result 2}, we have
\begin{align}
\lambda^{W}Y\cong (\lambda^WR)\Lotimes Y\cong (L^{\mathbb{W}}R)\otimes_RY.
\label{tensor Cech computation}
\end{align}

Let $F$ be a flat $R$-module and $M$ be a finitely generated $R$-module. Then we see from Lemma \ref{tensor representation} that 
$$(\bar{\lambda}^{W_i}F)\otimes_RM\cong \bar{\lambda}^{W_i}(F\otimes_RM).$$
This fact ensures that 
$(\bar{\lambda}^{(i_m,\ldots i_1, i_0)}R)\otimes_RM\cong \bar{\lambda}^{(i_m,\ldots i_1, i_0)}M$.
Thus, if $Y$ is a complex of finitely generated $R$-modules, then there is a natural isomorphism 
\begin{align}
(L^{\mathbb{W}}R)\otimes_RY\cong \tot L^{\mathbb{W}}Y
\label{tot formula}
\end{align}
in $C(\Mod R)$.
By (\ref{tensor Cech computation}) and (\ref{tot formula}), we have shown the following proposition.

\begin{proposition}\label{finitely generated cor 2}
Let $W$ be a subset of $\Spec R$ and $\mathbb{W}=\{W_i\}_{0\leq i\leq n}$ be a system of slices of $W$.
Let $Y$ be a complex of finitely generated $R$-modules.
Then there is an isomorphism in $\cD$;
$$\lambda^WY\cong\tot L^{\mathbb{W}}Y.$$
Under this identification, $\ell^{\mathbb{W}}Y:Y\to\tot L^{\mathbb{W}}Y$ coincides with $\eta^{W}Y:Y\to \lambda^WY$ in $\cD$.
\end{proposition}\medskip

We see from (\ref{tot formula}) and the remark below that it is also possible to give a quick proof of this proposition, provided that $Y$ is a right bounded complex of finitely generated $R$-modules.

\begin{remark}\label{left derived functor K(Mod R)}
Let $W$ be a subset of $\Spec R$ and $\mathbb{W}=\{W_i\}_{0\leq i\leq n}$ be a system of slices of $W$.
We denote by $K(\Mod R)$ the homotopy category of complexes of $R$-modules.
Note that $\tot L^{\mathbb{W}}$ induces a triangulated functor $K(\Mod R)\to K(\Mod R)$, which we also write $\tot L^{\mathbb{W}}$. Then it is seen from Corollary \ref{cor main result 2} that $\lambda^{W}:\cD\to \cD$ is isomorphic to the left derived functor of $\tot L^{\mathbb{W}}:K(\Mod R)\to K(\Mod R)$.
\end{remark}

Let $W$ be a subset of $\Spec R$ such that $n=\dim W$ is finite.
By Proposition \ref{finitely generated cor 2}, if an $R$-module $M$ is finitely generated, then $\lambda^WM\in \cD^{[0,n]}$.
On the other hand, since $\lambda^{V(\fa)}\cong \LLambda^{V(\fa)}$ for an ideal $\fa$, it can happen that $H^i(\lambda^{W}M)\neq 0$ for some $i<0$ when $M$ is not finitely generated, see \cite[Example 5.3]{NY}.

\begin{remark}
Let $n\geq 0$ be an integer. 
Let $\fa_i$ be ideals of $R$ and $S_i$ be multiplicatively closed subsets of $R$ for $0\leq i\leq n$.
In Notation \ref{Cech notation} and Definition \ref{Cech def}, 
one can replace $\bar{\lambda}^{(i)}=\bar{\lambda}^{W_i}$  by $\Lambda^{V(\fa_i)}(-\otimes_RS_i^{-1}R)$, and  construct a kind of \v{C}ech complexes.
For the \v{C}ech complex and $\lambda^W$ with $W=\bigcup_{0\leq i\leq n}(V(\fa_i)\cap U_{S_i})$, it is possible to show similar results to Corollary \ref{cor main result 2} and Proposition \ref{finitely generated cor 2}, provided that one of the following conditions holds: 
(1) $V(\fa_i)\cap U_{S_i}$ is specialization-closed in $\bigcup_{i\leq j\leq n}(V(\fa_{j})\cap U_{S_{j}})$ for each $0\leq i\leq n$;  
(2) $V(\fa_i)\cap U_{S_i}$ is generalization-closed in $\bigcup_{0\leq j\leq i}(V(\fa_{j})\cap U_{S_{j}})$ for each $0\leq i\leq n$.
\end{remark}

\section{\v{C}ech complexes and complexes of pure-injective modules}

In this section, as an application, we give a functorial way to construct a quasi-isomorphism from a complex of flat $R$-modules or a complex of finitely generated $R$-modules to a complex of pure-injective $R$-modules.\medskip

We start with the following well-known fact.

\begin{lemma}\label{flat and pure-injective}
Let $X$ be a complex of flat $R$-modules and $Y$ be a complex of cotorsion $R$-modules. 
We assume that one of the following conditions holds:
\begin{enumerate}[label={\rm(\arabic*)}]
\item $X$ is a right bounded and $Y$ is left bounded;
\item $X$ is bounded and $\dim R$ is finite.
\end{enumerate}
Then we have $\RHom_R(X,Y)\cong \Hom_R(X,Y)$.
\end{lemma}

One can prove this lemma by \cite[Theorem 12.5.4]{KS} and Theorem \ref{RG}.\medskip

Next, we recall the notion of pure-injective modules and resolutions. 
We say that a morphism $f:M\to N$ of $R$-modules is pure if $f\otimes_R L$ is a monomorphism in $\Mod R$ for any $R$-module $L$.
Moreover an $R$-module $P$ is called pure-injective if $\Hom_R(f,P)$ is an epimorphism in $\Mod R$ for any pure morphism $f:M\to N$ of $R$-modules.
Clearly, all injective $R$-modules are pure-injective.
Furthermore, all pure-injective $R$-modules are cotorsion, see \cite[Lemma 5.3.23]{EJ}.

Let $M$ be an $R$-module.
A complex $P$ together with a quasi-isomorphism $M\to P$ is called a pure-injective resolution of $M$ if $P$ consists of pure-injective $R$-modules and $P^i=0$ for $i<0$. 
It is known that any $R$-module has a minimal pure-injective resolution, which is constructed by using pure-injective envelopes, see \cite{E2} and \cite[Example 6.6.5, Definition 8.1.4]{EJ}. Moreover, if $F$ is a flat $R$-module and $P$ is  a pure-injective resolution of $M$, then we have $\RHom_R(F, M)\cong\Hom_R(F,P)$ by Lemma \ref{flat and pure-injective}.\medskip

Now we observe that any cotorsion flat $R$-module is pure-injective.
Consider an $R$-module of the form $(\bigoplus_{B}R_\fp)^\wedge_{\fp}$ with some index set $B$ and a prime ideal $\fp$, which is a cotorsion flat $R$-module. Writing $E_R(R/\fp)$ for the injective hull of $R/\fp$, we have
$$(\bigoplus_{B}R_\fp)^\wedge_{\fp}\cong \Hom_R(E_R(R/\fp),\bigoplus_{B}E_R(R/\fp)),$$
see \cite[Theorem 3.4.1]{EJ}.
It follows from tensor-hom adjunction that $\Hom_R(M, I)$ is pure-injective for any $R$-module $M$ and any injective $R$-module $I$.
Hence $(\bigoplus_{B}R_\fp)^\wedge_{\fp}$ is pure-injective.
Thus any cotorsion flat $R$-module is pure-injective, see Proposition \ref{Enochs theorem}.

There is another example of pure-injective $R$-modules. 
Let $M$ be a finitely generated $R$-module.
Using Five Lemma, we are able to prove an isomorphism 
$$\Hom_R\big(E_R(R/\fp),\bigoplus_{B}E_R(R/\fp)\big)\otimes_RM\cong \Hom_R\big(\Hom_R(M,E_R(R/\fp)),\bigoplus_{B}E_R(R/\fp)\big).$$
Therefore $(\bigoplus_{B}R_\fp)^\wedge_{\fp}\otimes_RM$ is pure-injective; it is also isomorphic to $(\bigoplus_{B}M_\fp)^\wedge_{\fp}$ by Lemma \ref{tensor representation}.
Moreover, Proposition \ref{commute with tensor} implies that $\cosupp (\bigoplus_{B}M_\fp)^\wedge_{\fp}\subseteq \{\fp\}$.
\bigskip

By the above observation, we see that Corollary \ref{[a,b] general}, (\ref{tot formula}) and Proposition \ref{finitely generated cor 2} yield the following theorem, which is one of the main results of this paper.

\begin{theorem}\label{pure-injective resolution}
Let $W$ be a subset of $\Spec R$ and $\mathbb{W}=\{W_i\}_{0\leq i\leq n}$ be a system of slices of $W$.
Let $Z$ be a complex  of flat $R$-modules or a complex of  finitely generated $R$-modules.
We assume that $\cosupp Z\subseteq W$.
Then $\ell^{\mathbb{W}}Z:Z\to \tot L^{\mathbb{W}}Z$ is a quasi-isomorphism, where $\tot L^{\mathbb{W}}Z$ consists of pure-injective $R$-modules with cosupports in $W$. 
\end{theorem}

\begin{remark}\label{minimal compare}
Let $N$ be a flat or finitely generated $R$-module.
Suppose that $d=\dim R$ is finite.
Set $W_{i}=\Set{\fp\in \Spec R | \dim R/\fp= i}$ and $\mathbb{W}=\{W_i\}_{0\leq i\leq d}$.
By Theorem \ref{pure-injective resolution}, we obtain a pure-injective resolution $\ell^{\mathbb{W}}N:N\to L^{\mathbb{W}}N$ of $N$, that is, there is an exact sequence of $R$-modules of the following form:
$$0\to N\to \prod_{0\leq i_0\leq d}\bar{\lambda}^{(i_0)}N\to\prod_{0\leq i_0<i_1\leq d}\bar{\lambda}^{(i_1,i_0)}N\to\cdots\to \bar{\lambda}^{(d,\ldots, 0)}N\to 0$$
We remark that,  in $C(\Mod R)$, $L^\mathbb{W}N$ need not be isomorphic to a minimal pure-injective resolution $P$ of $N$.
In fact, when $N$ is a projective or finitely generated $R$-module, it holds that $P^0\cong \prod_{\fm\in W_0}\widehat{N_\fm}= \bar{\lambda}^{(0)}N$ (cf. \cite[Theorem 3]{W} and \cite[Remark 6.7.12]{EJ}), while $(L^\mathbb{W}N)^0=\prod_{0\leq i_0\leq d}\bar{\lambda}^{(i_0)}N$.
Furthermore, Enochs \cite[Theorem 2.1]{E2} proved that if $N$ is flat $R$-module, then $P^i$ is of the form $\prod_{\fp\in W_{\geq i}} T^i_\fp$ for $0\leq i\leq d$ (cf. Notation \ref{Enochs rep}), where $W_{\geq i}=\Set{\fp\in \Spec R | \dim R/\fp\geq i}$.

On the other hand, for a flat or finitely generated $R$-module $N$, the differential maps in the pure-injective resolution $L^{\mathbb{W}}N$ are concretely described.
In addition, our approach based on the localization functor $\lambda^W$ and the \v{C}ech complex $L^{\mathbb{W}}$ provide a natural morphism $\ell^{\mathbb{W}}:\id_{C(\Mod R)}\to \tot L^{\mathbb{W}}$ which induces isomorphisms in $\cD$ for all complexes of flat $R$-modules and complexes of finitely generated $R$-modules.
The reader should also compare Theorem \ref{pure-injective resolution} with \cite[Proposition 5.9]{PT}.
\end{remark}\smallskip

We close this paper with the following example of Theorem \ref{pure-injective resolution}.

\begin{example}
Let $R$ be a 2-dimensional local domain with  quotient field $Q$.
Let $\mathbb{W}=\{W_i\}_{0\leq i\leq 2}$ be as in Remark \ref{minimal compare}.
Then $L^\mathbb{W}R$ is a pure-injective resolution of $R$, and $L^\mathbb{W}R$ is of the following form:
$$
0\to \disp{Q\oplus (\prod_{\fp\in W_1}\widehat{R_\fp})\oplus\widehat{R}}\to (\prod_{\fp \in W_1}\widehat{R_\fp})_{(0)}\oplus (\widehat{R})_{(0)}\oplus\prod_{\fp\in W_1}\widehat{(\widehat{R})_\fp} \to  
(\prod_{\fp\in W_1}\widehat{(\widehat{R})_\fp})_{(0)}
\to 0$$
\end{example}

\bibliographystyle{amsplain}

\begin{thebibliography}{99}
\bibitem{AJL} L. Alonso Tarr\' io, A. Jerem\' ias L\' opez and J. Lipman, {\it Local homology and cohomology on schemes}, Ann. Scient. \' Ec. Norm. Sup. {\bf 30} (1997), 1--39. Correction, available at https://www.math.purdue.edu/\verb|~|lipman/papers/homologyfix.pdf

\bibitem{BIK} D. Benson, S. Iyengar, and H. Krause, {\it Local cohomology and support for triangulated categories}, Ann. Scient. \' Ec. Norm. Sup. (4) {\bf 41} (2008), 1--47.

\bibitem{BIK3} D. Benson, S. Iyengar, and H. Krause, {\it Colocalizing subcategories and cosupport}, J. reine angew. Math. {\bf 673} (2012), 161--207.

\bibitem{BN} M. B\"okstedt and A. Neeman, {\it Homotopy limits in triangulated categories}, Compositio Math. {\bf 86} (1993), 209--234.

\bibitem{Bourbaki} N. Bourbaki, {\it Alg{\`e}bre commutative}, Springer-Verlag (2006).

\bibitem{E1}E. E. Enochs, {\it Flat covers and cotorsion flat modules}, Proc. Amer. Math. Soc., {\bf 92} (1984), 179--184.

\bibitem{E2} E. E. Enochs, {\it Minimal pure-injective resolutions of flat modules}, J. Algebra, {\bf105} (1987), 351--364.

\bibitem{EJ} E. E. Enochs and O. M. G. Jenda. {\it Relative Homological Algebra}, De Gruyter Expositions in Mathematics, {\bf 30}, Walter De Gruyter (2000).

\bibitem{FI} H.-B. Foxby and S. Iyengar, {\it Depth and amplitude for unbounded complexes}, Commutative algebra and its interactions with algebraic geometry (Grenoble-Lyon 2001), Contemp. Math. {\bf 331}, American Math. Soc. Providence, RI (2003), 119--137.

\bibitem{GR} R. Gordon and J. C. Robson, {\it Krull dimension}, Memoirs Amer. Math. Soc., {\bf 133} (1973).


\bibitem{GM} J. P. C. Greenlees and J. P. May, {\it Derived functors of $I$-adic completion and local homology}, J. Algebra {\bf 149} (1992), 438--453.

\bibitem{Ha} R. Hartshorne, {\it Residues and Duality: Lecture Notes of a Seminar on the Work of A. Grothendieck}, Lecture Notes in Math. {\bf 20}, Springer-Verlag (1966).

\bibitem{KS} M. Kashiwara and P. Schapira, {\it Categories and sheaves}, Grundlehren der Mathematischen Wissenschaften {\bf 332}, Springer-Verlag (2006).


\bibitem{K} H. Krause {\it Localization theory for triangulated categories, Triangulated categories}, London Math. Soc. Lecture Note Ser. {\bf 375},161--235. Cambridge Univ. Press (2010).

\bibitem{L} J. Lipman, {\it Lectures on local cohomology and duality, Local cohomology and its applications}, Lect. Notes Pure Appl. Math. {\bf 226} (2002), Dekker, New York, 39--89.


\bibitem{Nag} M. Nagata, {\it Local rings}, Robert E. Krieger Publishing Company, Huntington, New York (1975).

\bibitem{NY} T. Nakamura and Y. Yoshino, {\it A Local duality principle in derived categories of commutative Noetherian rings}, to appear in J. Pure Appl. Algebra.

\bibitem{N} A. Neeman, {\it The chromatic tower of $D(R)$}, Topology {\bf 31} (1992), 519--532.


\bibitem{N3} A. Neeman, {\it Colocalizing subcategories of D(R)}, J. reine angew. Math. {\bf 653} (2011), 221--243.


\bibitem{RG} M. Raynaud and L. Gruson, {\it Crit\`eres de platitude et de projectivit\'e}, Invent. math. {\bf 13} (1971), 1--89.

\bibitem{SW} S. Sather-Wagstaff and R. Wicklein, {\it Support and adic finiteness for complexes}, Comm. Algebra, {\bf 45}  (2017), 2569--2592.

\bibitem{Si} A.-M. Simon, {\it Some homological properties of complete modules}, Math. Proc. Camb. Phil. Soc. {\bf 108} (1990), 231--246.

\bibitem{PT1} P. Thompson, {\it Cosupport computations for finitely generated modules over commutative noetherian rings}, arXiv:1702.03270.

\bibitem{PT} P. Thompson, {\it Minimal complexes of cotorsion flat modules}, arXiv:1702.02985v2.

\bibitem {W} R. B. Warfield, Jr. {\it Purity and algebraic compactness for modules}, Pacific J. of Math. {\bf 28} (1969), 699--719.

\bibitem{X} J. Xu, {\it Flat Covers of Modules}, Lecture Notes in Math. {\bf 1634}, Springer-Verlag, Berlin Heidelberg (1996). 

 \end{thebibliography}

\end{document}